\documentclass[a4paper]{amsart}
\pdfoutput=1
\usepackage[utf8]{inputenc}
\usepackage[T1]{fontenc}
\usepackage{amssymb,mathtools}
\usepackage{mathrsfs}
\usepackage{dsfont}
\usepackage{lmodern}

\usepackage[all]{xy}
\usepackage{microtype}

\usepackage{enumitem}
\setlist[enumerate,1]{label=\textup{(\arabic*)}}

\usepackage[lite]{amsrefs}

\renewcommand*{\PrintDOI}[1]{\href{http://dx.doi.org/\detokenize{#1}}{doi: \detokenize{#1}}}

\BibSpec{book}{%
  +{}  {\PrintPrimary}                {transition}
  +{,} { \textit}                     {title}
  +{.} { }                            {part}
  +{:} { \textit}                     {subtitle}
  +{,} { \PrintEdition}               {edition}
  +{}  { \PrintEditorsB}              {editor}
  +{,} { \PrintTranslatorsC}          {translator}
  +{,} { \PrintContributions}         {contribution}
  +{,} { }                            {series}
  +{,} { \voltext}                    {volume}
  +{,} { }                            {publisher}
  +{,} { }                            {organization}
  +{,} { }                            {address}
  +{,} { \PrintDateB}                 {date}
  +{,} { }                            {status}
  +{}  { \parenthesize}               {language}
  +{}  { \PrintTranslation}           {translation}
  +{;} { \PrintReprint}               {reprint}
  +{.} { }                            {note}
  +{.} {}                             {transition}
  +{} { \PrintDOI}                   {doi}
  +{} { available at \url}            {eprint}
  +{}  {\SentenceSpace \PrintReviews} {review}
}

\usepackage[pdftitle={A bicategorical interpretation for relative Cuntz-Pimsner algebras},
pdfauthor={Camila Fabre Sehnem and Ralf Meyer},
pdfsubject={Mathematics}
]{hyperref}

\numberwithin{equation}{section}

\theoremstyle{plain}
\newtheorem{thm}[equation]{Theorem}
\newtheorem{cor}[equation]{Corollary}
\newtheorem{lem}[equation]{Lemma}
\newtheorem{prop}[equation]{Proposition}

\theoremstyle{definition}
\newtheorem{defn}[equation]{Definition}

\theoremstyle{remark}
\newtheorem{rem}[equation]{Remark}
\newtheorem{example}[equation]{Example}

\newcommand{\ZZ}{\mathbb{Z}}

\newcommand{\NN}{\mathbb{N}}
\newcommand{\TT}{\mathbb{T}}
\newcommand{\CC}{\mathbb{C}}

\DeclareMathOperator{\obj}{ob}

\newcommand*{\nb}{\nobreakdash}
\newcommand*{\Star}{\(^*\)\nobreakdash-}
\newcommand{\Cst}{\mathrm{C}^*}
\newcommand{\idealin}{\mathrel{\triangleleft}} 
\newcommand*{\Bound}{\mathbb B}
\newcommand*{\Comp}{\mathbb K}
\newcommand*{\defeq}{\mathrel{\vcentcolon=}}
\newcommand{\Hilm}[1][E]{\mathcal{#1}}
\newcommand{\Toep}{\mathcal{T}}
\newcommand{\CP}{\mathcal{O}}
\newcommand{\Corr}{\mathfrak{C}}
\newcommand{\Bim}{\mathrm{pr,*}}
\newcommand{\proper}{\mathrm{pr}}
\newcommand{\id}{\mathrm{id}}
\newcommand{\op}{\mathrm{op}}

\DeclarePairedDelimiter{\norm}{\lVert}{\rVert}
\DeclarePairedDelimiter{\bra}{\langle}{\rvert}
\DeclarePairedDelimiter{\ket}{\lvert}{\rangle}
\DeclarePairedDelimiterX{\braket}[2]{\langle}{\rangle}{#1\,\delimsize\vert\,\mathopen{}#2}
\DeclarePairedDelimiterX{\BRAKET}[2]{\langle}{\rangle}{\!\delimsize\langle#1\,\delimsize\vert\,\mathopen{}#2\delimsize\rangle\!}

\DeclarePairedDelimiterX{\setgiven}[2]{\{}{\}}{#1\,{:}\,\mathopen{}#2}

\begin{document}
\title[A bicategorical interpretation for Cuntz--Pimsner algebras]{A bicategorical interpretation for\\ relative Cuntz--Pimsner algebras}

\author{Ralf Meyer}
\email{rmeyer2@uni-goettingen.de}

\author{Camila F. Sehnem}
\email{camila.fabre-sehnem@mathematik.uni-goettingen.de}

\address{Mathematisches Institut, Georg-August-Universität Göttingen,
  Bunsenstraße 3--5, 37073, Göttingen, Germany}


\thanks{The second author was supported by CNPq (Brazil).}

\begin{abstract}
  We interpret the construction of relative Cuntz--Pimsner algebras of
  correspondences in terms of the correspondence bicategory, as a
  reflector into a certain sub-bicategory.  This generalises a
  previous characterisation of absolute Cuntz--Pimsner algebras of
  proper correspondences as colimits in the correspondence bicategory.
\end{abstract}
\maketitle

\section{Introduction}

Many important \(\Cst\)\nb-algebras
may be described as (relative) Cuntz--Pimsner algebras, see
\cites{Pimsner:Generalizing_Cuntz-Krieger, Muhly-Solel:Tensor,
  Katsura:Cstar_correspondences}.  These are defined by
triples~\((A,\Hilm,J)\),
where~\(A\)
is a \(\Cst\)\nb-algebra,
\(\Hilm\)
is a \(\Cst\)\nb-correspondence
from~\(A\)
to itself, that is, a Hilbert \(A\)\nb-module
with a nondegenerate left action of~\(A\)
by adjointable operators, \(\varphi\colon A\to\Bound(\Hilm)\),
and \(J\idealin A\)
is an ideal that acts on~\(\Hilm\)
by compact operators, that is, \(\varphi(J)\subseteq \Comp(\Hilm)\).
The Cuntz--Pimsner covariance condition is only required on~\(J\).

We view the correspondence~\(\Hilm\)
as a generalised endomorphism of~\(A\).
If~\(\Hilm\)
comes from an automorphism~\(\alpha\)
of~\(A\),
then the relative Cuntz--Pimsner algebra for \(J= A\)
is naturally isomorphic to the crossed product
\(A\rtimes_\alpha \ZZ\).
So we may view Cuntz--Pimsner algebras as analogues of crossed
products for automorphisms.  This is made precise
in~\cite{Albandik-Meyer:Colimits} by viewing both crossed products and
Cuntz--Pimsner algebras as colimits of diagrams in the bicategory of
\(\Cst\)\nb-correspondences.
The interpretation of Cuntz--Pimsner algebras
in~\cite{Albandik-Meyer:Colimits} is limited, however, to
\emph{proper} correspondences, that is,
\(\varphi(A)\subseteq \Comp(\Hilm)\),
and the ``absolute'' case \(J=A\).
This article is concerned with another bicategorical interpretation of
the Cuntz--Pimsner algebra construction, which needs no properness and
extends to the relative case.

Our results use the equivalence between \(\Cst\)\nb-algebras
with a \(\TT\)\nb-action
and Fell bundles over~\(\ZZ\),
see~\cite{Abadie-Eilers-Exel:Morita_bimodules}.
The spectral decomposition of a \(\TT\)\nb-action~\(\beta\)
on a \(\Cst\)\nb-algebra~\(B\)
gives a Fell bundle~\((B_n)_{n\in\NN}\)
over the group~\(\ZZ\)
whose section \(\Cst\)\nb-algebra
\(\Cst((B_n)_{n\in\NN})\) is canonically isomorphic to~\(B\); namely,
\[
B_n \defeq \setgiven{b\in B}{\beta_z(b) = z^n \cdot b}
\]
for \(n\in\ZZ\)
with the multiplication, involution and norm from~\(B\).
Conversely, the section \(\Cst\)\nb-algebra
of any Fell bundle over~\(\ZZ\)
carries a canonical gauge action of~\(\TT\).
The Fell bundle underlying a Cuntz--Pimsner algebra is
\emph{semi-saturated}, that is, \(B_n\cdot B_m = B_{n+m}\)
if \(n,m\ge0\)
(or if \(n,m\le0\)).
Here and below, \(X\cdot Y\)
means the closed linear span of
\(\setgiven{x\cdot y}{x\in X,\ y\in Y}\).
By the results of~\cite{Abadie-Eilers-Exel:Morita_bimodules}, a
semi-saturated Fell bundle is determined by its fibres \(B_0\)
and~\(B_1\):
\(B_0\)
is a \(\Cst\)\nb-algebra,
\(B_1\)
is a Hilbert \(B_0\)\nb-bimodule,
and the crossed product for the Hilbert \(B_0\)\nb-bimodule~\(B_1\)
is isomorphic to the section \(\Cst\)\nb-algebra
of the Fell bundle generated by \(B_0\) and~\(B_1\).

Thus we split the construction of Cuntz--Pimsner algebras with their
canonical \(\TT\)\nb-action
into two steps.  The first builds the Hilbert bimodule
\(\CP^1_{J,\Hilm}\)
over~\(\CP^0_{J,\Hilm}\),
the second takes the crossed product for this Hilbert bimodule.  When
we include the gauge action, then the second step is reversible using
the spectral decomposition.  This article interprets the first step in
the construction as a reflector to a sub-bicategory.  A Hilbert
bimodule is a \(\Cst\)\nb-correspondence
with an additional left inner product, which is unique if it exists.
Thus Hilbert bimodules form a full sub-bicategory in the
correspondence bicategory.  We describe a bicategory whose objects are
the triples \((A,\Hilm,J)\)
needed to define a relative Cuntz--Pimsner algebra.  Those triples
where~\(\Hilm\)
is a Hilbert bimodule and~\(J\)
is Katsura's ideal for~\(\Hilm\)
form a full sub-bicategory.  We show that the construction of
\((\CP^0_{J,\Hilm},\CP^1_{J,\Hilm})\)
is a \emph{reflector} onto this sub-bicategory.  Roughly speaking, a
reflector approximates a given object by an object in the
sub-bicategory in the optimal way.  More precisely, it is a left
(bi)adjoint to the inclusion of the sub-bicategory.

We gradually work up to such bicategorical considerations.
Section~\ref{sec:prelim} deals with known properties of relative
Cuntz--Pimsner algebras.  We also discuss their Fell bundle structure
coming from the gauge action, and we show that the Cuntz--Pimsner
algebra~\(\CP_{J,\Hilm}\)
is the crossed product of its gauge-fixed point
algebra~\(\CP^0_{J,\Hilm}\)
by the Hilbert \(\CP^0_{J,\Hilm}\)-bimodule
\(\CP^1_{J,\Hilm}\).
Section~\ref{sec:prelim} culminates in a result about the
functoriality of relative Cuntz--Pimsner algebras, which goes back to
an idea of Schweizer~\cite{Schweizer:Crossed_Cuntz-Pimsner}.  We
correct his idea and extend it to the relative case by defining proper
covariant correspondences between triples \((A,\Hilm,J)\)
so that they induce correspondences between the associated relative
Cuntz--Pimsner algebras.

This construction is upgraded in Section~\ref{mainsection} to a
homomorphism of bicategories (or ``functor'') from a certain
bicategory~\(\Corr^\NN_\proper\)
to the \(\TT\)\nb-equivariant
correspondence bicategory~\(\Corr^\TT\).
The objects of~\(\Corr^\NN_\proper\)
are the triples \((A,\Hilm,J)\)
needed to define a relative Cuntz--Pimsner algebra, the arrows are the
proper covariant correspondences introduced in
Section~\ref{sec:prelim}, and the \(2\)\nb-arrows
are isomorphisms of covariant correspondences.  Whereas Schweizer
reduces to ordinary categories by identifying isomorphic
correspondences, bicategories are crucial for our purposes, as
in~\cite{Albandik-Meyer:Colimits}.

Then we define a sub-bicategory~\(\Corr^\NN_\Bim\)
by restricting to Hilbert bimodules instead of correspondences.  We
prove a crucial statement about covariant correspondences, namely,
that proper covariant correspondences
\((A,\Hilm,J)\to (B,\Hilm[G],I_{\Hilm[G]})\)
are ``equivalent'' to proper covariant correspondences
\((\CP^0_{J,\Hilm},\CP^1_{J,\Hilm},I_{\Hilm})\to
(B,\Hilm[G],I_{\Hilm[G]})\)
for all \((B,\Hilm[G],I_{\Hilm[G]})\)
in~\(\Corr^\NN_\Bim\),
that is, for a Hilbert \(B\)\nb-bimodule~\(\Hilm[G]\)
and Katsura's ideal \(I_{\Hilm[G]}\).

Section~\ref{mainsection1} introduces the bicategorical language to
understand this fact: it says that a certain arrow
\[
\upsilon_{(A,\Hilm,J)}\colon (A,\Hilm,J)\to(\CP^0_{J,\Hilm},\CP^1_{J,\Hilm},I_{\Hilm})
\]
is a universal arrow from \((A,\Hilm,J)\)
to~\(\Corr^\NN_\Bim\).
The existence of universal arrows implies an adjunction
(see~\cite{Fiore:Pseudo_biadjoints}).  So general bicategory theory
upgrades the ``equivalence'' observed above to our main statement,
namely, that the sub-bicategory
\(\Corr^\NN_\Bim\subseteq\Corr^\NN_\proper\)
is reflective and that the reflector homomorphism
\(\Corr^\NN_\proper\to\Corr^\NN_\Bim\)
acts on objects by mapping \((A,\Hilm,J)\)
to \((\CP^0_{J,\Hilm},\CP^1_{J,\Hilm},I_{\CP^1_{J,\Hilm}})\).
We describe this reflector in detail and show that its composite with
the crossed product homomorphism \(\Corr^\NN_\Bim\to\Corr^\TT\)
is the relative Cuntz--Pimsner algebra homomorphism
\(\Corr^\NN_\Bim\subseteq\Corr^\TT\)
described in Section~\ref{mainsection}.  The
definitions of bicategories, homomorphisms, transformations, and
modifications are recalled in the appendix, together with some
examples related to the correspondence bicategory.

\section{Preliminaries}
\label{sec:prelim}

In this section, we recall basic results on Cuntz--Pimsner algebras,
and their gauge action and Fell bundle structure.  We correct and
generalise an idea by Schweizer on the functoriality of Cuntz--Pimsner
algebras for covariant correspondences.

\subsection{Correspondences}

Let \(\Hilm[F]_1\),
\(\Hilm[F]_2\)
be Hilbert \(B\)-modules.
Let \(\Bound(\Hilm[F]_1,\Hilm[F]_2)\)
be the space of adjointable operators from \(\Hilm[F]_1\)
to \(\Hilm[F]_2\).
Let \(\ket{\xi}\bra{\eta} \in \Bound(\Hilm[F]_1,\Hilm[F]_2)\)
for \(\xi\in\Hilm[F]_2\)
and \(\eta\in\Hilm[F]_1\)
be the \emph{generalised rank\nb-\(1\)
  operator} defined by
\(\ket{\xi}\bra{\eta}(\zeta)\defeq \xi\braket{\eta}{\zeta}_B\).
Let \(\Comp(\Hilm[F]_1,\Hilm[F]_2)\)
be the closed linear span of \(\ket{\xi}\bra{\eta}\)
for \(\xi\in\Hilm[F]_1\)
and \(\eta\in\Hilm[F]_2\).
Elements of \(\Comp(\Hilm[F]_1,\Hilm[F]_2)\)
are called \emph{compact operators}.  We abbreviate
\(\Bound(\Hilm[F]) \defeq \Bound(\Hilm[F],\Hilm[F])\)
and \(\Comp(\Hilm[F]) \defeq \Comp(\Hilm[F],\Hilm[F])\)
if \(\Hilm[F]=\Hilm[F]_1=\Hilm[F]_2\).

\begin{lem}
  \label{lem:Comp_submodules}
  Let \(\Hilm_1\subseteq \Hilm[F]_1\)
  and \(\Hilm_2\subseteq \Hilm[F]_2\)
  be Hilbert \(B\)\nb-submodules.
  There is a unique map
  \(\Comp(\Hilm_1,\Hilm_2) \to \Comp(\Hilm[F]_1,\Hilm[F]_2)\)
  that maps \(\ket{\xi}\bra{\eta}\in\Comp(\Hilm_1,\Hilm_2)\)
  to \(\ket{\xi}\bra{\eta}\in\Comp(\Hilm[F]_1,\Hilm[F]_2)\)
  for all \(\xi\in\Hilm_2\),
  \(\eta\in\Hilm_1\).  This map is injective.
\end{lem}

\begin{defn}
  Let \(A\)
  and~\(B\)
  be \(\Cst\)\nb-algebras.
  A \emph{correspondence} from \(A\)
  to~\(B\)
  is a Hilbert \(B\)\nb-module~\(\Hilm[F]\)
  with a \emph{nondegenerate} left action of~\(A\)
  through a \Star{}homomorphism
  \(\varphi\colon A\to\Bound(\Hilm[F])\).
  A correspondence is \emph{proper} if
  \(\varphi(A)\subseteq\Comp(\Hilm[F])\).
  It is \emph{faithful} if~\(\varphi\)
  is injective.  We write \(\Hilm[F]\colon A\leadsto B\)
  to say that~\(\Hilm[F]\) is a correspondence from~\(A\) to~\(B\).
\end{defn}

\begin{defn}
  A \emph{Hilbert \(A,B\)-bimodule}
  is a (right) Hilbert \(B\)\nb-module~\(\Hilm[F]\)
  with a \emph{left} Hilbert \(A\)\nb-module
  structure \(\BRAKET{{\cdot}}{{\cdot}}_A\)
  such that \(\BRAKET{\xi}{\eta}_A\zeta = \xi\braket{\eta}{\zeta}_B\)
  for all \(\xi, \eta, \zeta\in \Hilm[F]\).
\end{defn}

If~\(\Hilm[F]\)
is a Hilbert \(A,B\)-bimodule,
then~\(A\)
acts by adjointable operators on~\(\Hilm[F]\)
and~\(B\)
acts by adjointable operators for the left Hilbert
\(A\)\nb-module
structure, that is,
\(\BRAKET{\xi b}{\eta}_A = \BRAKET{\xi}{\eta b^*}_A\)
for all \(\xi,\eta\in\Hilm[F]\)
and all \(b\in B\).
In particular, \(\Hilm\)
is an \(A,B\)-bimodule.
The next lemma characterises which correspondences may be enriched to
Hilbert bimodules:

\begin{lem}[see \cite{Echterhoff-Kaliszewski-Quigg-Raeburn:Categorical}*{Example~1.6}]
  \label{characterization-hilbertbimodule}
  A correspondence \(\Hilm[F]\colon A\leadsto B\)
  carries a Hilbert \(A,B\)-bimodule
  structure if and only if there is an ideal \(I\idealin A\)
  such that left action on~\(\Hilm[F]\)
  restricts to a \Star{}isomorphism \(I \cong \Comp(\Hilm[F])\).
  In this case, the ideal~\(I\)
  and the left inner product are unique, and
  \(I = \BRAKET{\Hilm[F]}{\Hilm[F]}_A\).
\end{lem}

\begin{defn}
  Let \(\Hilm[F]_1,\Hilm[F]_2\colon A\leadsto B\)
  be \(\Cst\)\nb-correspondences.
  A \emph{correspondence isomorphism}
  \(\Hilm[F]_1\Rightarrow \Hilm[F]_2\)
  is a unitary \(A, B\)-bimodule
  isomorphism from~\(\Hilm[F]_1\)
  to~\(\Hilm[F]_2\).
  We write~``\(\Rightarrow\)''
  because these isomorphisms are the \(2\)\nb-arrows
  in bicategories that we are going to construct.
\end{defn}

Let~\(\Hilm[F]\)
be a Hilbert \(B\)\nb-module
and let \(\varphi\colon A\to \Bound(\Hilm[F])\)
be a \Star{}homomorphism.  For \(\xi\in\Hilm\), we define an operator
\[
T_\xi\colon \Hilm[F]\to \Hilm\otimes_\varphi\Hilm[F],\qquad
\eta \mapsto \xi\otimes\eta.
\]
It is adjointable with
\(T_\xi^* (\zeta\otimes\eta)=\varphi(\braket{\xi}{\zeta})\eta\)
on elementary tensors, see~\cite{Pimsner:Generalizing_Cuntz-Krieger}.
Hence
\[
T_\xi T_\zeta^*=\ket{\xi}\bra{\zeta}\otimes1,\qquad
T_\zeta^*T_\xi=\varphi(\braket{\zeta}{\xi}),
\]
where \(\ket{\xi}\bra{\zeta}\otimes1\)
is the image of \(\ket{\xi}\bra{\zeta}\)
under the canonical map
\(\Bound(\Hilm)\to \Bound(\Hilm\otimes_\varphi\Hilm[F])\),
\(T\mapsto T\otimes1\).
Hence the operator~\(T_\xi\)
for \(\xi\in\Hilm\)
is compact if and only if \(\varphi(\braket{\xi}{\xi}) = T_\xi^* T_\xi\)
is compact.

\begin{lem}[\cite{Pimsner:Generalizing_Cuntz-Krieger}*{Corollary 3.7}]
  \label{lem:Ttens_compact}
  Let \(J\defeq \varphi^{-1}(\Comp(\Hilm[F]))\idealin A\)
  and let \(T\in\Comp(\Hilm)\).
  The operator \(T\otimes1\)
  on \(\Hilm\otimes_A \Hilm[F]\)
  is compact if and only if \(T\in\Comp(\Hilm\cdot J)\)
  \textup{(}see Lemma~\textup{\ref{lem:Comp_submodules}} for the
  inclusion \(\Comp(\Hilm\cdot J) \subseteq \Comp(\Hilm)\)\textup{)}.
\end{lem}

In particular, if \(\varphi(A)\subseteq\Comp(\Hilm[F])\),
then \(T\otimes 1 \in \Comp(\Hilm\otimes_{\varphi}\Hilm[F])\)
for all \(T\in\Comp(\Hilm)\).

\subsection{\texorpdfstring{$\Cst$\nb-}{C*-}algebras of correspondences}

Let \(\Hilm\colon A\leadsto A\)
be a correspondence over~\(A\).
Let \(\varphi\colon A\to\Bound(\Hilm)\)
be the left action.  Let \(\Hilm^{\otimes n}\)
be the \(n\)\nb-fold
tensor product of~\(\Hilm\)
over~\(A\).
By convention, \({\Hilm}^{\otimes 0}\defeq A\).
Let \(\Hilm^+\defeq \bigoplus^\infty_{n=0} \Hilm^{\otimes n}\)
be the \emph{Fock space} of~\(\Hilm\),
see~\cite{Pimsner:Generalizing_Cuntz-Krieger}.  Define
\[
t_\xi^n\colon\Hilm^{\otimes n}\to\Hilm^{\otimes n+1},\qquad
\eta\mapsto \xi\otimes\eta,
\]
for \(n\ge0\)
and \(\xi\in\Hilm\);
this is the operator~\(T_\xi\)
above for \(\Hilm[F]=\Hilm^{\otimes n}\).
The operators~\(t_\xi^n\)
combine to an operator \(t_\xi\in\Bound(\Hilm^+)\),
that is, \(t_\xi|_{\Hilm^{\otimes n}} = t_\xi^n\).
Let \(\varphi_\infty\colon A\to \Bound(\Hilm^+)\)
be the obvious representation by block diagonal operators and let
\(t_\infty\colon \Hilm\to \Bound(\Hilm^+)\)
be the linear map \(\xi\mapsto t_\xi\).

\begin{defn}
  The \emph{Toeplitz \(\Cst\)\nb-algebra}~\(\Toep_{\Hilm}\)
  of~\(\Hilm\)
  is the \(\Cst\)\nb-subalgebra
  of \(\Bound(\Hilm^+)\)
  generated by \(\varphi_\infty(A) + t_\infty(\Hilm)\).
\end{defn}

Let~\(J\)
be an ideal of \(A\)
with \(\varphi(J)\subseteq \Comp(\Hilm)\).
Let~\(P_0\)
be the projection in \(\Bound(\Hilm^+)\)
that is the identity on \(A\subseteq\Hilm^+\)
and that vanishes on~\(\Hilm^{\otimes n}\)
for \(n\ge1\).
Then \(J_0\defeq \varphi_\infty(J)P_0\)
is contained in~\(\Toep_{\Hilm}\).
The ideal in~\(\Toep_{\Hilm}\)
generated by~\(J_0\)
is equal to \(\Comp(\Hilm^+J) \subseteq \Comp(\Hilm^+)\).

\begin{defn}[\cite{Muhly-Solel:Tensor}*{Definition 2.18}]
  The \emph{relative Cuntz--Pimsner algebra}~\(\CP_{J,\Hilm}\)
  of~\(\Hilm\)
  with respect to~\(J\) is \(\Toep_{\Hilm}/\Comp(\Hilm^+J)\).
\end{defn}

The following three cases are particularly important.  First, if
\(J=\{0\}\),
then \(\CP_{J,\Hilm}\)
is the Toeplitz \(\Cst\)\nb-algebra~\(\Toep_{\Hilm}\).
Secondly, if \(J=\varphi^{-1}(\Comp(\Hilm))\)
and~\(\varphi\)
is injective, then~\(\CP_{J,\Hilm}\)
is the algebra~\(\tilde{\CP}_{\Hilm}\)
defined by Pimsner~\cite{Pimsner:Generalizing_Cuntz-Krieger}.
Third, if~\(J\) is \emph{Katsura's ideal}
\begin{equation}
  \label{eq:Katsura_ideal}
  I_{\Hilm}\defeq \varphi_{\Hilm}^{-1}(\Comp(\Hilm))\cap (\ker\varphi_{\Hilm})^\perp,  
\end{equation}
then~\(\CP_{I_{\Hilm},\Hilm}\)
is Katsura's Cuntz--Pimsner algebra as defined
in~\cite{Katsura:Cstar_correspondences}.

\begin{prop}
  \label{pro:Katsura_injective}
  Katsura's ideal~\(I_{\Hilm}\)
  in~\eqref{eq:Katsura_ideal} is the largest ideal~\(J\)
  in~\(A\)
  with \(\varphi(J)\subseteq \Comp(\Hilm)\)
  for which the canonical map \(A\to \CP_{J,\Hilm}\)
  is injective.
\end{prop}

\begin{proof}
  That~\(\pi_{I_{\Hilm}}\)
  is injective is \cite{Katsura:Cstar_correspondences}*{Proposition
    4.9}.  The ideal~\(I_{\Hilm}\)
  is maximal with this property because any ideal \(J\idealin A\)
  with \(\varphi(J)\subseteq \Comp(\Hilm)\)
  and \(J\not\subseteq (\ker \varphi)^\perp\)
  must contain \(a\in J\)
  with \(\varphi(a)=0\).
  Then \(\varphi_\infty(a) \in \varphi_\infty(J)\cdot P_0\)
  becomes~\(0\) in \(\CP_{J,\Hilm}\).
\end{proof}

\begin{defn}
  Let \(\Hilm\colon A\leadsto A\)
  be a correspondence and~\(B\)
  a \(\Cst\)\nb-algebra.
  A \emph{representation} of~\(\Hilm\)
  in~\(B\)
  is a pair~\((\pi, t)\),
  where \(\pi\colon A\to B\)
  is a \Star{}homomorphism, \(t\colon\Hilm\to B\) is a linear map, and
  \begin{enumerate}
  \item \(\pi(a) t(\xi) = t(\varphi(a) \xi)\)
    for all \(a\in A\) and \(\xi\in \Hilm\);
  \item \(t(\xi)^* t(\eta) = \varphi(\braket{\xi}{\eta}_A)\)
    for all \(\xi, \eta\in \Hilm\).
  \end{enumerate}
  These conditions imply \(t(\xi) \pi(a) = t(\xi a)\)
  for all \(\xi\in \Hilm\) and \(a\in A\).
\end{defn}

In particular, \((\varphi_\infty,t_\infty)\)
is a representation of~\(\Hilm\)
in the Toeplitz \(\Cst\)\nb-algebra~\(\Toep_{\Hilm}\).
This representation is universal in the following sense:

\begin{prop}
  \label{universaltoeplitz}
  Any representation~\((\pi, t)\)
  of~\(\Hilm\)
  in a \(\Cst\)\nb-algebra~\(B\)
  is of the form
  \((\tilde{\pi}\circ\varphi_\infty,\tilde{\pi}\circ t_\infty)\)
  for a unique \Star{}homomorphism
  \(\tilde{\pi}\colon\Toep_{\Hilm}\to B\).
  Conversely,
  \((\tilde{\pi}\circ\varphi_\infty,\tilde{\pi}\circ t_\infty)\)
  is a representation of~\(\Hilm\)
  for any \Star{}homomorphism
  \(\tilde{\pi}\colon\Toep_{\Hilm}\to B\).
\end{prop}

\begin{lem}
  For any representation \((\pi, t)\)
  of~\(\Hilm\),
  there is a unique \Star{}homomorphism
  \(\pi^1\colon\Comp(\Hilm)\to B\)
  with \(\pi^1(\ket{\xi}\bra{\eta}) = t_\xi t^*_\eta\)
  for all \(\xi, \eta\in\Hilm\).
\end{lem}

\begin{prop}[\cite{Muhly-Solel:Tensor}*{Theorem 2.19}]
  \label{universalpropertycp}
  The representation~\(\tilde{\pi}\)
  of~\(\Toep_{\Hilm}\)
  associated to a representation \((\pi, t)\)
  of \(\Hilm\)
  factors through the quotient~\(\CP_{J,\Hilm}\)
  of~\(\Toep_{\Hilm}\) if and only if
  \begin{equation}
    \label{covariancecondition}
    \pi(a)=\pi^1(\varphi(a))\qquad
    \text{for all }a\in J.
  \end{equation}
  In this case, we call the representation \emph{covariant on~\(J\).}
\end{prop}

Let \((\pi_J, t_J)\)
be the canonical representation of~\(\Hilm\)
in~\(\CP_{J,\Hilm}\).
Proposition~\ref{universalpropertycp} says that \((\pi_J, t_J)\)
is the universal representation of~\(\Hilm\)
that is covariant on~\(J\).

\begin{prop}
  A representation~\((\pi,t)\)
  in~\(B\)
  is covariant on~\(J\)
  if and only if \(\pi(J) \subseteq t(\Hilm)\cdot B\).
  \label{pro:covariant_through_subspaces}
\end{prop}

\begin{proof}
  Let \(a\in J\).
  Then \(\pi^1(\varphi(a))\)
  is contained in the closed linear span of \(t(\Hilm)t(\Hilm)^*\)
  and hence in \(t(\Hilm)\cdot B\).
  So \(\pi(a) \in t(\Hilm)\cdot B\)
  is necessary for \(\pi(a) = \pi^1(\varphi(a))\).
  Conversely, assume \(\pi(a) \in t(\Hilm)\cdot B\)
  for all \(a\in J\).
  We have
  \(\pi(a)\cdot t(\xi) = t(\varphi(a)\xi) = \pi^1(\varphi(a)) t(\xi)\)
  for all \(\xi\in\Hilm\)
  (see \cite{Katsura:Cstar_correspondences}*{Lemma 2.4}).  Hence
  \((\pi(a)-\pi^1(\varphi(a)))\cdot t(\Hilm)\cdot B=0\).
  Since \(\pi(a^*),\pi^1(\varphi(a^*)) \in t(\Hilm)\cdot B\),
  we get
  \((\pi(a)-\pi^1(\varphi(a)))\cdot (\pi(a)-\pi^1(\varphi(a)))^*=0\).
  This is equivalent to \(\pi(a) = \pi^1(\varphi(a))\).
\end{proof}

\subsection{Gauge action and Fell bundle structure}

Let \(\Hilm\colon A\leadsto A\)
be a correspondence and let \(J\idealin A\)
be an ideal with \(\varphi(J) \subseteq \Comp(\Hilm)\).
If \((\pi,t)\)
is a representation of~\(\Hilm\)
that is covariant on~\(J\),
then so is \((\pi,z\cdot t)\)
for \(z\in\TT\).
This operation on representations comes from an automorphism of the
relative Cuntz--Pimsner algebra~\(\CP_{J,\Hilm}\)
by its universal property.  These automorphisms define a continuous
action~\(\gamma\)
of~\(\TT\)
on~\(\CP_{J,\Hilm}\), called the \emph{gauge action}.  Let
\[
\CP^n_{J,\Hilm}\defeq \setgiven{b\in\CP_{J,\Hilm}}
{\gamma_z(b) = z^nb \text{ for all }z\in\TT}
\]
for \(n\in\ZZ\)
be the \(n\)th
spectral subspace.  These spectral subspaces form a Fell bundle
over~\(\ZZ\),
that is,
\(\CP^n_{J,\Hilm} \cdot \CP^m_{J,\Hilm} \subseteq
\CP^{n+m}_{J,\Hilm}\)
and \((\CP^n_{J,\Hilm})^* = \CP^{-n}_{J,\Hilm}\)
for all \(n,m\in\ZZ\).
In particular, for \(J=\{0\}\)
we get a gauge action on~\(\Toep_{\Hilm}\)
and corresponding spectral subspaces
\(\Toep^n_{\Hilm} \subseteq \Toep_{\Hilm}\).
Explicitly, the gauge action on~\(\Toep_{\Hilm}\)
comes from the obvious \(\NN\)\nb-grading
on~\(\Hilm^+\):
if \(x\in\Toep_{\Hilm}\),
then \(x\in\Toep^n_{\Hilm}\)
if and only if \(x(\Hilm^{\otimes k}) \subseteq \Hilm^{\otimes n+k}\)
for all \(k\in\NN\);
this means \(x|_{\Hilm^{\otimes k}} = 0\)
if \(k+n<0\).
And~\(\CP_{\Hilm,J}^n\)
is the image of~\(\Toep_{\Hilm}^n\) in~\(\CP_{\Hilm,J}\).

\begin{lem}
  \label{fibers}
  Let \(n\in\ZZ\).
  The subspace~\(\CP^n_{J,\Hilm}\)
  in~\(\CP_{J,\Hilm}\)
  is the closed linear span of
  \(t_J(\xi_1)t_J(\xi_2)\dotsm t_J(\xi_k)\cdot t^*_J(\eta_l)\dotsm
  t^*_J(\eta_2)t^*_J(\eta_1)\)
  for \(\xi_i,\eta_j\in \Hilm\), \(k-l=n\).  If \(n\in\NN\), then
  \[
  \CP^n_{J,\Hilm}\cong \Hilm^{\otimes n}\otimes_A\CP^0_{J,\Hilm}
  \]
  as a correspondence \(A\leadsto \CP^0_{J,\Hilm}\).
  The Fell bundle \((\CP^k_{J,\Hilm})_{k\in\ZZ}\)
  is \emph{semi-saturated}, that is,
  \(\CP^k_{J,\Hilm} \cdot \CP^l_{J,\Hilm}=\CP^{k+l}_{J,\Hilm}\)
  if \(k,l\ge0\).
\end{lem}

\begin{proof}
  Let \(b\in\CP^n_{J,\Hilm}\)
  and let \(\epsilon>0\).
  Then~\(b\)
  is \(\epsilon\)\nb-close
  to a finite linear combination~\(b_\epsilon\)
  of monomials
  \(t_J(\xi_1)t_J(\xi_2)\dotsm t_J(\xi_k)\cdot t^*_J(\eta_l)\dotsm
  t^*_J(\eta_2)t^*_J(\eta_1)\) with \(k,l\in\NN\).  Define
  \[
  p_n(x) \defeq \int_{\TT} z^{-n} \gamma_z(x) \,\mathrm{d}z,\qquad
  x\in\CP_{J,\Hilm}.
  \]
  This is a contractive projection from~\(\CP_{J,\Hilm}\)
  onto~\(\CP^n_{J,\Hilm}\).
  Since \(p_n(b)=b\)
  and \(\norm{p_n}\le1\),
  we have \(\norm{b-p_n(b_\epsilon)}\le\epsilon\)
  as well.  Inspection shows that \(p_n\)
  maps a monomial
  \(t_J(\xi_1)t_J(\xi_2)\dotsm t_J(\xi_k)\cdot t^*_J(\eta_l)\dotsm
  t^*_J(\eta_2)t^*_J(\eta_1)\)
  to itself if \(k-l=n\)
  and kills it otherwise.  Hence~\(\CP^n_{J,\Hilm}\)
  is the closed linear span of such monomials with \(k-l=n\).

  The monomials generating~\(\CP^{k+l}_{J,\Hilm}\)
  for \(k,l\ge0\)
  are obviously in \(\CP^k_{J,\Hilm} \cdot \CP^l_{J,\Hilm}\).
  Hence the first statement immediately implies the last one.  There
  is an isometric \(A,\CP^0_{J,\Hilm}\)-bimodule map
  \[
  \Hilm^{\otimes n} \otimes_A\CP^0_{J,\Hilm} \to \CP^n_{J,\Hilm},\qquad
  \xi_1\otimes\dotsb \otimes \xi_n \otimes y\mapsto
  t_J(\xi_1)\dotsm t_J(\xi_n) \cdot y.
  \]
  The first statement implies that its image is dense, so it is
  unitary.
\end{proof}

The Fell bundle \((\CP^n_{J,\Hilm})_{n\in\ZZ}\)
need not be saturated, that is,
\(\CP^n_{J,\Hilm} \cdot \CP^{-n}_{J,\Hilm}\)
may differ from~\(\CP^0_{J,\Hilm}\).

\begin{thm}
  \label{the:CP_Hilbi_crossed}
  The relative Cuntz--Pimsner algebra is \(\TT\)\nb-equivariantly
  isomorphic to the crossed product of~\(\CP^0_{J,\Hilm}\)
  by the Hilbert \(\CP^0_{J,\Hilm}\)-bimodule
  \(\CP^1_{J,\Hilm}\)
  and to the full or reduced section \(\Cst\)\nb-algebra
  of the Fell bundle \((\CP^n_{J,\Hilm})_{n\in\ZZ}\).
\end{thm}

\begin{proof}
  The Fell bundle \((\CP^n_{J,\Hilm})_{n\in\ZZ}\)
  is semi-saturated by Lemma~\ref{fibers}.  Now the results of
  Abadie--Eilers--Exel~\cite{Abadie-Eilers-Exel:Morita_bimodules}
  imply our claims.
\end{proof}

Theorem~\ref{the:CP_Hilbi_crossed} splits the construction of relative
Cuntz--Pimsner algebras into two steps.  The first builds the Hilbert
\(\CP^0_{J,\Hilm}\)-bimodule
\(\CP^1_{J,\Hilm}\),
the second takes the crossed product for this Hilbert bimodule.  A
Hilbert bimodule~\(\Hilm[G]\)
on a \(\Cst\)\nb-algebra~\(B\)
is the same as a Morita--Rieffel equivalence between two ideals
in~\(B\)
or, briefly, a partial Morita--Rieffel equivalence on~\(B\)
(this point of view is explained
in~\cite{Buss-Meyer:Actions_groupoids}).  The crossed product
\(\CP^0_{J,\Hilm}\rtimes \CP^1_{J,\Hilm}\)
generalises the partial crossed product for a partial automorphism.
Many results about crossed products for automorphisms extend to
Hilbert bimodule crossed products.  In particular, the standard
criteria for simplicity and detection and separation of ideals are
extended in~\cite{Kwasniewski-Meyer:Aperiodicity}.

\begin{prop}
  \label{katsurasalgebra}
  The following conditions are equivalent:
  \begin{enumerate}
  \item \label{katsurasalgebra_1}%
    the map \(\pi_J\colon A\to\CP^0_{J,\Hilm}\)
    is an isomorphism;

  \item \label{katsurasalgebra_2}%
    the map \(\varphi\colon J\to\Comp(\Hilm)\) is an isomorphism;

  \item \label{katsurasalgebra_3}%
    the correspondence~\(\Hilm\)
    comes from a Hilbert bimodule and \(J=I_{\Hilm}\).
  \end{enumerate}
\end{prop}

\begin{proof}
  If \(J=I_{\Hilm}\)
  is Katsura's ideal, then everything follows from
  \cite{Katsura:Cstar_correspondences}*{Proposition~5.18}.  So it
  remains to observe that \ref{katsurasalgebra_1} and
  \ref{katsurasalgebra_2} fail if \(J\neq I_{\Hilm}\).
  Lemma~\ref{characterization-hilbertbimodule} shows that~\(\Hilm\)
  comes from a Hilbert bimodule if and only if there is an ideal~\(I\)
  in~\(A\)
  so that \(\varphi|_I\colon I \to \Comp(\Hilm)\)
  is an isomorphism.  In this case, \(I\)
  is the largest ideal on which~\(\varphi\)
  restricts to an injective map into~\(\Comp(\Hilm)\).
  So \(I=I_{\Hilm}\).
  Thus \ref{katsurasalgebra_2}\(\iff\)\ref{katsurasalgebra_3}.

  If \(J\not\subseteq I_{\Hilm}\),
  then \(A\to\CP_{J,\Hilm}\)
  is not injective by Proposition~\ref{pro:Katsura_injective}.
  So~\ref{katsurasalgebra_1} implies \(J\subseteq I_{\Hilm}\).
  If \(J\subseteq I_{\Hilm}\)
  and~\ref{katsurasalgebra_1} holds, then the map
  \(A\to\CP_{I_{\Hilm},\Hilm}\)
  is still surjective because \(\CP_{I_{\Hilm},\Hilm}\)
  is a quotient of~\(\CP_{J,\Hilm}\),
  and it is also injective by Proposition~\ref{pro:Katsura_injective}.
  Hence \(\CP_{I_{\Hilm},\Hilm} = \CP_{J,\Hilm}\).
  This implies \(\Comp(\Hilm^+ I_{\Hilm}) = \Comp(\Hilm^+ J)\)
  and hence \(I_{\Hilm} = J\)
  because of the direct summand~\(A\) in~\(\Hilm^+\).
\end{proof}

\begin{prop}
  \label{pro:Hilbi_CP}
  Let~\(\Hilm[G]\)
  be a Hilbert \(B\)\nb-bimodule
  and let~\(I_{\Hilm[G]}\)
  be Katsura's ideal for~\(\Hilm[G]\).
  Then \(\CP_{I_{\Hilm[G]},\Hilm[G]} \cong B\rtimes \Hilm[G]\)
  \(\TT\)\nb-equivariantly.
\end{prop}

\begin{proof}
  Theorem~\ref{the:CP_Hilbi_crossed} identifies
  \(\CP_{I_{\Hilm[G]},\Hilm[G]} \cong
  \CP^0_{I_{\Hilm[G]},\Hilm[G]}\rtimes
  \CP^1_{I_{\Hilm[G]},\Hilm[G]}\).
  Proposition~\ref{katsurasalgebra} gives
  \(B\cong \CP^0_{I_{\Hilm[G]},\Hilm[G]}\),
  and the isomorphism
  \(\CP^1_{I_{\Hilm[G]},\Hilm[G]}\cong \Hilm[G] \otimes_A
  \CP^0_{I_{\Hilm[G]},\Hilm[G]}\)
  from Lemma~\ref{fibers} implies that
  \(\Hilm[G] \cong \CP^1_{I_{\Hilm[G]},\Hilm[G]}\)
  as a Hilbert \(B\)\nb-bimodule.
\end{proof}

\subsection{Functoriality of relative Cuntz--Pimsner algebras}
\label{sec:CP_functorial}

Schweizer~\cite{Schweizer:Crossed_Cuntz-Pimsner} has defined
``covariant homomorphisms'' and ``covariant correspondences'' between
self-correspondences and has asserted that they induce
\Star{}homomorphisms and correspondences between the associated
Toeplitz and absolute Cuntz--Pimsner algebras.  For the proof of
functoriality for covariant correspondences he refers to a preprint
that never got published.  In fact, there are some technical pitfalls.
We correct his statement here, and also add a condition to treat
relative Cuntz--Pimsner algebras.

Throughout this subsection, let \(\Hilm\colon A\leadsto A\)
and \(\Hilm[G]\colon B\leadsto B\)
be correspondences and let \(J_A\subseteq \varphi^{-1}(\Comp(\Hilm))\)
and \(J_B\subseteq \varphi^{-1}(\Comp(\Hilm[G]))\) be ideals.

\begin{defn}
  \label{def:covariant_corr}
  A \emph{covariant correspondence} from \((A,\Hilm,J_A)\)
  to \((B, \Hilm[G],J_B)\)
  is a pair~\((\Hilm[F],V)\),
  where~\(\Hilm[F]\)
  is a correspondence \(A\leadsto B\)
  with \(J_A\cdot \Hilm[F] \subseteq \Hilm[F]\cdot J_B\)
  and~\(V\)
  is a correspondence isomorphism
  \(\Hilm\otimes_A\Hilm[F] \Rightarrow \Hilm[F]\otimes_B\Hilm[G]\).
  A covariant correspondence is \emph{proper} if~\(\Hilm[F]\) is proper.
\end{defn}

\begin{prop}
  \label{pro:CP_functorial}
  A proper covariant correspondence \((\Hilm[F],V)\)
  from \((A,\Hilm,J_A)\)
  to \((B, \Hilm[G],J_B)\)
  induces a proper \(\TT\)\nb-equivariant
  correspondence
  \(\CP_{\Hilm[F],V}\colon
  \CP_{J_A,\Hilm}\leadsto\CP_{J_B,\Hilm[G]}\).
\end{prop}

Schweizer~\cite{Schweizer:Crossed_Cuntz-Pimsner} claims this also for
non-proper correspondences, and he allows~\(V\)
to be a non-adjointable isometry.  In fact, a pair \((\Hilm[F],V)\)
where~\(V\)
is only a non-adjointable isometry induces a correspondence between
the Toeplitz \(\Cst\)\nb-algebras.  It is unclear, however, when this
correspondence descends to one between the absolute or relative
Cuntz--Pimsner algebras.  And we need~\(\Hilm[F]\)
to be proper.  Alternatively, we may require~\(\Hilm\)
instead of~\(\Hilm[F]\)
to be proper.  This situation is treated
in~\cite{Albandik-Meyer:Colimits}.

\begin{proof}
  We use the canonical \Star{}homomorphism
  \(\pi_{J_B}\colon B\to\CP_{J_B,\Hilm[G]}\)
  to view~\(\CP_{J_B,\Hilm[G]}\)
  as a proper correspondence \(B\leadsto \CP_{J_B,\Hilm[G]}\).
  Thus \(\Hilm[F]_\CP \defeq \Hilm[F] \otimes_B \CP_{J_B,\Hilm[G]}\)
  becomes a proper correspondence \(A\leadsto \CP_{J_B,\Hilm[G]}\),
  that is, a Hilbert \(\CP_{J_B,\Hilm[G]}\)-module
  with a representation \(\pi\colon A\to\Comp(\Hilm[F]_\CP)\).
  The \(\TT\)\nb-action
  on~\(\CP_{J_B,\Hilm[G]}\)
  induces a \(\TT\)\nb-action
  on \(\Hilm[F]_\CP\)
  because \(\pi_{J_B}(B) \subseteq\CP_{J_B,\Hilm[G]}^0\).
  We are going to define a map
  \(t\colon \Hilm\to\Comp(\Hilm[F]_\CP)\)
  such that \((\pi,t)\)
  is a representation of~\((A,\Hilm)\)
  on~\(\Hilm[F]_\CP\)
  that is covariant on~\(J_A\).
  Then Proposition~\ref{universalpropertycp} yields a representation
  \(\tilde\pi\colon \CP_{J_A,\Hilm}\to\Comp(\Hilm[F]_\CP)\).
  This is the desired correspondence
  \(\CP_{J_A,\Hilm}\leadsto\CP_{J_B,\Hilm[G]}\).

  There is an isometry
  \(\mu_{\Hilm[G]}\colon \Hilm[G] \otimes_B \CP_{J_B,\Hilm[G]}
  \Rightarrow \CP_{J_B,\Hilm[G]}\),
  \(\zeta\otimes y\mapsto t_\infty(\zeta)\cdot y\),
  of correspondences \(B\leadsto \CP_{J_B,\Hilm[G]}\).
  Usually, it is not unitary.  We define an isometry
  \[
  V^!\colon \Hilm \otimes_A \Hilm[F]_\CP
  = \Hilm \otimes_A \Hilm[F] \otimes_B \CP_{J_B,\Hilm[G]}
  \xRightarrow{V\otimes1} \Hilm[F] \otimes_B \Hilm[G] \otimes_B \CP_{J_B,\Hilm[G]}
  \xRightarrow{1\otimes\mu_{\Hilm[G]}} \Hilm[F] \otimes_B \CP_{J_B,\Hilm[G]}
  = \Hilm[F]_\CP.
  \]
  It yields a map~\(t\)
  from~\(\Hilm\)
  to the space of bounded operators on~\(\Hilm[F]_\CP\)
  by \(t(\xi)(\eta)\defeq V^!(\xi\otimes\eta)\).
  To show that \(t(\xi)\)
  is adjointable, we need that~\(\Hilm[F]_\CP\)
  is a proper correspondence \(A\leadsto \CP_{J_B,\Hilm[G]}\):
  then \(T_\xi\in\Comp(\Hilm[F]_\CP,\Hilm\otimes_A \Hilm[F]_\CP)\),
  and composition with~\(V^!\)
  maps this into \(\Comp(\Hilm[F]_\CP)\)
  by Lemma~\ref{lem:Comp_submodules}.  So even
  \(t(\xi)\in\Comp(\Hilm[F]_\CP)\) for all \(\xi\in\Hilm\).

  We claim that the pair~\((\pi,t)\)
  is a representation.  We have \(\pi(a) t(\xi) = t(\varphi(a)\xi)\)
  because~\(V^!\)
  is a left \(A\)\nb-module
  map.  And \(t(\xi_1)^* t(\xi_2) = \pi(\braket{\xi_1}{\xi_2})\)
  holds because
  \begin{multline*}
    \braket{t(\xi_1)\eta_1}{t(\xi_2)\eta_2}
    =  \braket{V^!(\xi_1\otimes\eta_1)}{V^!(\xi_2 \otimes\eta_2)}
    \\=  \braket{\xi_1\otimes\eta_1}{\xi_2 \otimes\eta_2}
    =  \braket{\eta_1}{\pi(\braket{\xi_1}{\xi_2})\eta_2}.
  \end{multline*}
  If \(J_A=0\),
  then we are done at this point, and we have not yet used that~\(V\)
  is unitary.  So the Toeplitz \(\Cst\)\nb-algebra
  of a correspondence remains functorial for proper covariant
  correspondences where~\(V\) is not unitary.

  It remains to prove that~\(\pi\)
  is covariant on~\(J_A\).
  By Proposition~\ref{pro:covariant_through_subspaces}, this is
  equivalent to
  \(\pi(J_A)(\Hilm[F]_\CP) \subseteq t(\Hilm)(\Hilm[F]_\CP)\).
  And
  \(J_B\cdot \CP_{J_B,\Hilm[G]} \subseteq t_{J_B}(\Hilm[G])\cdot
  \CP_{J_B,\Hilm[G]}\)
  holds because the canonical representation of~\((B,\Hilm[G])\)
  on \(\CP_{J_B,\Hilm[G]}\)
  is covariant on~\(J_B\).
  Since \(J_A\cdot\Hilm[F] \subseteq \Hilm[F]\cdot J_B\)
  by assumption,
  \[
  J_A\cdot\Hilm[F]_\CP
  \subseteq \Hilm[F]\otimes J_B\cdot \CP_{J_B,\Hilm[G]}
  \subseteq \Hilm[F]\otimes t_{J_B}(\Hilm[G])\cdot \CP_{J_B,\Hilm[G]}
  = (1\otimes \mu_{\Hilm[G]})(\Hilm[F]\otimes_B \Hilm[G] \otimes_B \CP_{J_B,\Hilm[G]}).
  \]
  Since~\(V\)
  is unitary, we may rewrite this further as
  \(V^!(\Hilm[E]\otimes_A \Hilm[F] \otimes_B \CP_{J_B,\Hilm[G]}) =
  t(\Hilm[E]) \cdot \Hilm[F]_\CP\).
  This finishes the proof that~\((\pi,t)\)
  is covariant on~\(J_A\).
  The operators \(t(\xi)\)
  for \(\xi\in\Hilm\)
  are homogeneous of degree~\(1\)
  for the \(\TT\)\nb-action.
  Thus~\(\tilde\pi\) is \(\TT\)\nb-equivariant.
\end{proof}

\begin{example}
  \label{exa:Schweizer_counterexample}
  Let \(A=B\)
  and \(J= J_A=J_B\neq\{0\}\)
  and let \(\Hilm\subseteq\Hilm[G]\)
  be an \(A\)\nb-invariant
  Hilbert submodule.  Then the identity correspondence \(\Hilm[F]=A\)
  with the inclusion map
  \(\Hilm\otimes_A \Hilm[F] \cong \Hilm \hookrightarrow \Hilm[G] \cong
  \Hilm[G] \otimes_B \Hilm[F]\)
  is a covariant correspondence in the notation of Schweizer.  There
  is indeed a canonical \Star{}homomorphism
  \(\Toep_{\Hilm} \to \Toep_{\Hilm[G]}\).
  But it need not descend to the relative Cuntz--Pimsner algebras
  because \(\varphi_{\Hilm[G]}(a)\in\Comp(\Hilm[G])\)
  for \(a\in J\)
  need not be the extension of \(\varphi_{\Hilm}(a)\in\Comp(\Hilm)\)
  given by Lemma~\ref{lem:Comp_submodules}.  So the Cuntz--Pimsner
  covariance conditions for \(\CP_{\Hilm,J}\)
  and~\(\CP_{\Hilm[G],J}\)
  may be incompatible.  We ask~\(V\)
  to be unitary to avoid this problem.
\end{example}

\begin{example}
  \label{exa:universalarrow}
  Turn~\(\CP^0_{J,\Hilm}\),
  into a proper \(\Cst\)\nb-correspondence
  \(A\leadsto \CP^0_{J,\Hilm}\)
  with the obvious left action of~\(A\).
  The proper correspondence
  \(\CP^0_{J,\Hilm}\colon A\leadsto \CP^0_{J,\Hilm}\)
  with the isomorphism from Lemma~\ref{fibers} is a proper covariant
  correspondence from \(\Hilm\colon A\leadsto A\)
  with the ideal~\(J\)
  to
  \(\CP^1_{J,\Hilm}\colon \CP^0_{J,\Hilm} \leadsto \CP^0_{J,\Hilm}\)
  with Katsura's ideal~\(I_{\CP^1_{J,\Hilm}}\).
  It remains to show that
  \(J\cdot \CP^0_{J,\Hilm} \subseteq \CP^0_{J,\Hilm}\cdot
  I_{\CP^1_{J,\Hilm}} = I_{\CP^1_{J,\Hilm}}\).
  Since~\(\CP^1_{J,\Hilm}\)
  is a Hilbert bimodule, Katsura's ideal is equal to the range ideal
  of the left inner product, that is, the closed linear span of
  \(x y^*\)
  for all \(x,y\in \CP^1_{J,\Hilm}\).
  This contains \(\Comp(\Hilm)\)
  for \(x,y\in\Hilm\),
  which in turn contains~\(J\)
  by the Cuntz--Pimsner covariance condition on~\(J\),
  see Proposition~\ref{universalpropertycp}.
  So \(J\cdot \CP^0_{J,\Hilm} \subseteq I_{\CP^1_{J,\Hilm}}\).
  The relative
  Cuntz--Pimsner algebra of
  \((\CP^0_{J,\Hilm},\CP^1_{J,\Hilm},I_{\CP^1_{J,\Hilm}})\)
  is again~\(\CP^0_{J,\Hilm}\)
  by Proposition~\ref{katsurasalgebra}.  The correspondence
  \(\CP^0_{J,\Hilm} \leadsto \CP^0_{J,\Hilm}\)
  associated to the covariant correspondence above is just the
  identity correspondence on~\(\CP^0_{J,\Hilm}\).
\end{example}

\begin{rem}
  \label{rem:covariance_simplifies_Bmax}
  If \(J_A=0\)
  or \(J_B=\varphi^{-1}(\Comp(\Hilm[G]))\),
  then the condition
  \(J_A\cdot \Hilm[F] \subseteq \Hilm[F] \cdot J_B\)
  for covariant transformations \((A,\Hilm,J_A) \to (B,\Hilm[G],J_B)\)
  always holds and so may be left out.  This is clear if \(J_A=0\).
  Let \(J_B = \varphi^{-1}(\Comp(\Hilm[G]))\).
  Since~\(\Hilm[F]\)
  is proper, \(J_A\)
  acts on
  \(\Hilm \otimes_A \Hilm[F] \cong \Hilm[F] \otimes_B \Hilm[G]\)
  by compact operators by Lemma~\ref{lem:Ttens_compact}.  Again by
  Lemma~\ref{lem:Ttens_compact}, this implies
  \(J_A \subseteq \Comp(\Hilm[F]\cdot J_B)\).
  Thus \(J_A\cdot \Hilm[F] \subseteq \Hilm[F] \cdot J_B\).
\end{rem}

\begin{example}
  Covariant correspondences are related to the \emph{\(T\)\nb-pairs}
  used by Katsura~\cite{Katsura:Ideal_structure_correspondences} to
  describe the ideal structure of relative Cuntz--Pimsner algebras.
  For this, we specialise to covariant correspondences out of
  \((A,\Hilm,J)\)
  where the underlying correspondence comes from a quotient map
  \(A\to A/I\).
  That is, \(\Hilm[F]= A/I\colon A\leadsto A/I\)
  for an ideal \(I\idealin A\).
  When is this part of a covariant correspondence from \((A,\Hilm,J)\)
  to \((A/I,\Hilm',J')\) for some \(\Hilm',J'\)?

  There are natural isomorphisms
  \(\Hilm\otimes_A \Hilm[F] \cong \Hilm/\Hilm I\)
  and \(\Hilm[F]\otimes_{A/I} \Hilm' \cong \Hilm'\)
  as correspondences \(A\leadsto A/I\).
  So the only possible choice for~\(\Hilm'\)
  is \(\Hilm' \defeq \Hilm/\Hilm I\)
  with a left \(A/I\)\nb-action
  which gives the canonical \(A\)\nb-action
  when composed with the quotient map \(A\to A/I\).
  Such a correspondence \(\Hilm/\Hilm I\colon A/I \leadsto A/I\)
  exists if and only if~\(\Hilm\)
  is \emph{positively invariant}, that is,
  \(I \Hilm \subseteq \Hilm I\).
  Assume this to be the case.

  An ideal \(J'\idealin A/I\)
  is equivalent to an ideal \(I'\idealin A\)
  that contains~\(I\).
  For a covariant correspondence, we require
  \(J \Hilm[F] \subseteq \Hilm[F] J'\),
  which means that \(J\subseteq I'\).
  And in order for \((A/I,\Hilm',J')\)
  to be an object of~\(\Corr^\NN_\proper\),
  the ideal~\(J'\)
  or, equivalently, \(I'\),
  should act by compact operators on \(\Hilm' \defeq \Hilm/\Hilm I\).

  Then there is an isomorphism
  \(\Hilm \otimes_A \Hilm[F] \cong \Hilm[F]\otimes_A \Hilm'\).
  It is unique up to an automorphism of~\(\Hilm/\Hilm I\),
  that is, a unitary operator on~\(\Hilm/\Hilm I\)
  that also commutes with the left action of \(A\)
  or~\(A/I\),
  but this shall not concern us.  So we get a covariant correspondence
  in this case.  This induces a correspondence from~\(\CP_{J, \Hilm}\)
  to~\(\CP_{J', \Hilm'}\)
  by Proposition~\ref{pro:CP_functorial}.  Actually, our covariant
  correspondence is a covariant homomorphism, and so the correspondence
  from Proposition~\ref{pro:CP_functorial} comes from a
  \(\TT\)\nb-equivariant
  \Star{}homomorphism, which turns out to be surjective.  So a pair of
  ideals~\((I,I')\)
  as above induces a \(\TT\)\nb-equivariant
  quotient or, equivalently, a \(\TT\)\nb-invariant
  ideal in~\(\CP_{J, \Hilm}\).

  Sometimes different pairs \((I,I')\)
  produce the same quotient of~\(\CP_{J,\Hilm}\).
  If~\(I'/I\)
  contains elements that act by~\(0\)
  on~\(\Comp(\Hilm/\Hilm I)\),
  then the map \(A/I \to \CP_{J',\Hilm'}\)
  is not injective by Proposition~\ref{pro:Katsura_injective}.  Then
  we may enlarge~\(I\)
  without changing the relative Cuntz--Pimsner algebra.  When we add
  the condition that no non-zero element of~\(I'/I\)
  acts by a compact operator on~\(\Hilm/\Hilm\cdot I\),
  then we get exactly the \emph{\(T\)-pairs}
  with \(J\subseteq I'\)
  of~\cite{Katsura:Ideal_structure_correspondences}.  The
  \(T\)\nb-pairs
  \((I, I')\)
  with \(J\subseteq I'\)
  correspond bijectively to gauge-invariant ideals
  of~\(\CP_{J, \Hilm}\)
  by
  \cite{Katsura:Ideal_structure_correspondences}*{Proposition~11.9}.
\end{example}

\section{Bicategories of correspondences and Hilbert bimodules}
\label{mainsection}

We are going to enrich the relative Cuntz--Pimsner algebra
construction to a homomorphism (functor) from a suitable bicategory of
covariant correspondences to the \(\TT\)\nb-equivariant
correspondence bicategory.  Most of the work is already done in
Proposition~\ref{pro:CP_functorial}, which describes how this
homomorphism acts on arrows.  It remains to define the appropriate
bicategories and write down the remaining data of a homomorphism.

The correspondence bicategory of \(\Cst\)\nb-algebras
and related bicategories have been discussed in
\cites{Buss-Meyer-Zhu:Non-Hausdorff_symmetries,
  Buss-Meyer-Zhu:Higher_twisted, Buss-Meyer:Actions_groupoids,
  Albandik-Meyer:Colimits}.  We recall basic bicategorical definitions
in the appendix for the convenience of the reader.  Here we go through
these notions much more quickly.  Let~\(\Corr\)
be the correspondence bicategory.  It has \(\Cst\)\nb-algebras
as objects, \(\Cst\)\nb-correspondences
as arrows, and correspondence isomorphisms as \(2\)\nb-arrows.
The composition is the tensor product~\(\otimes_B\)
of \(\Cst\)\nb-correspondences.

Given any bicategory~\(\mathfrak{D}\),
there is a bicategory \(\Corr^{\mathfrak{D}}\)
with homomorphisms \(\mathfrak{D}\to\Corr\)
as objects, transformations between these homomorphisms as arrows, and
modifications between these transformations as \(2\)\nb-arrows
(see the appendix for these notions).  There is also a continuous
version of this for a locally compact, topological
bicategory~\(\mathfrak{D}\).
In particular, we shall use the \(\TT\)\nb-equivariant
correspondence bicategory~\(\Corr^\TT\).
Its objects are \(\Cst\)\nb-algebras
with a continuous \(\TT\)\nb-action.
Its arrows are \(\TT\)\nb-equivariant
\(\Cst\)\nb-correspondences,
and \(2\)\nb-arrows
are \(\TT\)\nb-equivariant
isomorphisms of \(\Cst\)\nb-correspondences.

When~\(\mathfrak{D}\)
is the monoid~\((\NN,+)\),
we may simplify the bicategory \(\Corr^{\mathfrak{D}}\),
see \cite{Albandik-Meyer:Colimits}*{Section~5}.  An object in it is
equivalent to a \(\Cst\)\nb-algebra~\(A\)
with a self-correspondence \(\Hilm\colon A\leadsto A\).
An arrow is equivalent to a covariant correspondence (without the
condition \(J_A \Hilm[F]\subseteq \Hilm[F] J_B\)),
and a \(2\)\nb-arrow
is equivalent to an isomorphism between two covariant correspondences.
The bicategory~\(\Corr^\NN_\proper\)
that we need is a variant of~\(\Corr^\NN\)
where we add the ideal~\(J\)
and allow only proper covariant correspondences as arrows.

\begin{defn}
  \label{bicategory}
  The bicategory \(\Corr^\NN_\proper\)
  has the following data (see Definition~\ref{def:bicategory}):
  \begin{itemize}
  \item Objects are triples \((A, \Hilm,J)\),
    where \(A\)
    is a \(\Cst\)\nb-algebra,
    \(\Hilm\colon A\leadsto A\)
    is a \(\Cst\)\nb-correspondence,
    and \(J\subseteq\varphi^{-1}(\Comp(\Hilm))\) is an ideal.
  \item Arrows \((A, \Hilm, J)\to (A_1, \Hilm_1,J_1)\)
    are proper covariant correspondences \((\Hilm[F],u)\)
    from \((A,\Hilm,J)\)
    to \((A_1, \Hilm_1,J_1)\),
    that is, \(\Hilm[F]\)
    is a proper correspondence \(A\leadsto A_1\)
    with \(J\Hilm[F]\subseteq\Hilm[F] J_1\)
    and~\(u\)
    is a correspondence isomorphism
    \(\Hilm\otimes_A \Hilm[F] \Rightarrow \Hilm[F] \otimes_{A_1}
    \Hilm_1\).

  \item \(2\)\nb-Arrows
    \((\Hilm[F]_0,u_0) \Rightarrow (\Hilm[F]_1,u_1)\)
    are isomorphisms of covariant correspondences, that is,
    correspondence isomorphisms
    \(w\colon\Hilm[F]_0\Rightarrow \Hilm[F]_1\)
    for which the following diagram commutes:
    \[
    \xymatrix{
      \Hilm\otimes_A\Hilm[F]_0\ar@{=>}[r]^{u_0} \ar@{=>}[d]_{1_{\Hilm}\otimes w}&
      \Hilm[F]_0\otimes_{A_1}\Hilm_1\ar@{=>}[d]^{w\otimes 1_{\Hilm_1}} \\
      \Hilm\otimes_A\Hilm[F]_1 \ar@{=>}[r]^{u_1} &
      \Hilm[F]_1\otimes_{A_1}\Hilm_1
    }
    \]

  \item The vertical product of \(2\)\nb-arrows
    \[
    w_0\colon(\Hilm[F]_0,u_0)\Rightarrow(\Hilm[F]_1,u_1),\qquad
    w_1\colon(\Hilm[F]_1,u_1)\Rightarrow(\Hilm[F]_2,u_2)
    \]
    is the usual product
    \(w_1\cdot w_0\colon\Hilm[F]_0\to\Hilm[F]_2\).
    This is indeed a \(2\)\nb-arrow from
    \((\Hilm[F]_0,u_0)\) to \((\Hilm[F]_2,u_2)\).
    And the vertical product is associative and unital.  Thus the
    arrows \((A, \Hilm, J)\to (A_1, \Hilm_1,J_1)\)
    and the \(2\)\nb-arrows
    between them form a category
    \(\Corr^\NN_\proper((A, \Hilm, J), (A_1, \Hilm_1,J_1))\).

  \item Let \((\Hilm[F],u)\colon(A, \Hilm, J)\to(A_1, \Hilm_1, J_1)\)
    and
    \((\Hilm[F]_1,u_1)\colon(A_1, \Hilm_1, J_1)\to(A_2, \Hilm_2,
    J_2)\)
    be arrows.  Their product is
    \((\Hilm[F]_1,u_1)\circ(\Hilm[F],u)\defeq
    (\Hilm[F]\otimes_{A_1}\Hilm[F]_1, u\bullet u_1)\),
    where~\(u\bullet u_1\) is the composite correspondence isomorphism
    \[
    \Hilm\otimes_A\Hilm[F]\otimes_{A_1}\Hilm[F]_1
    \xrightarrow{u\otimes 1_{\Hilm[F]_1}}
    \Hilm[F]\otimes_{A_1}\Hilm_1\otimes_{A_1}\Hilm[F]_1
    \xrightarrow{1_{\Hilm[F]}\otimes u_1}
    \Hilm[F]\otimes_{A_1}\Hilm[F]_1\otimes_{A_2}\Hilm_2.
    \]

  \item The horizontal product for a diagram of arrows and
    \(2\)\nb-arrows
    \[
    \xymatrix{(A, \Hilm, J)\ar@/^1pc/[rr]^{(\Hilm[F],u)}="a"
      \ar@/^-1pc/[rr]_{(\widetilde{\Hilm[F]},\widetilde{u})}="b"&
      &(A_1, \Hilm_1, J_1) \ar@/^1pc/[rr]^{(\Hilm[F]_1,u_1)}="c"
      \ar@/^-1pc/[rr]_{(\widetilde{\Hilm[F]_1},\widetilde{u_1})}="d"
      & &(A_2, \Hilm_2, J_2) \ar@{=>}^{w}"a";"b" \ar@{=>}^{w_1}"c";"d"
    }
    \]
    is the \(2\)\nb-arrow
    \[
    \xymatrix{
      (A, \Hilm, J) \ar@/^1pc/[rrrr]^{(\Hilm[F]\otimes_{A_1}\Hilm[F]_1,u\bullet u_1)}="a"
      \ar@/^-1pc/[rrrr]_{(\widetilde{\Hilm[F]}\otimes_{A_1}\widetilde{\Hilm[F]_1},
        \widetilde{u}\bullet\widetilde{u_1})}="b" & & &&
      ( A_2, \Hilm_2, J_2).
      \ar@{=>}^{w\otimes w_1 }"a";"b"}
    \]
    This horizontal product and the product of arrows combine to
    composition bifunctors
    \begin{multline*}
      \Corr^\NN_\proper((A, \Hilm, J), (A_1, \Hilm_1,J_1))
      \times \Corr^\NN_\proper((A_1, \Hilm_1, J_1), (A_2, \Hilm_2,J_2))
      \\\to \Corr^\NN_\proper((A, \Hilm, J), (A_2, \Hilm_2,J_2)).
    \end{multline*}

  \item The unit arrow on the object~\((A,\Hilm,J)\)
    is the proper covariant correspondence \((A,\iota_{\Hilm})\),
    where~\(A\)
    is the identity correspondence, that is, \(A\)
    with the obvious \(A\)\nb-bimodule
    structure and the inner product \(\braket{x}{y}\defeq x^*y\),
    and~\(\iota_{\Hilm}\) is the canonical isomorphism
    \[
    \Hilm\otimes_A A \cong  \Hilm \cong  A\otimes_A \Hilm
    \]
    built from the right and left actions of~\(A\) on~\(\Hilm\).

  \item The associators and unitors are the same as in the
    correspondence bicategory.  Thus they inherit the coherence
    conditions needed for a bicategory.
  \end{itemize}
\end{defn}

\begin{thm}
  \label{the:CP_functor}
  There is a homomorphism \(\Corr^\NN_\proper\to\Corr^\TT\)
  that maps each object \((A,\Hilm,J)\)
  to its relative Cuntz--Pimsner algebra and is the construction of
  Proposition~\textup{\ref{pro:CP_functorial}} on arrows.
\end{thm}

\begin{proof}
  The construction in Proposition~\ref{pro:CP_functorial} is
  ``natural'' and thus functorial for isomorphisms of covariant
  correspondences, and it maps the identity covariant correspondence
  to the identity \(\TT\)\nb-equivariant
  correspondence on the relative Cuntz--Pimsner algebras.  Let
  \((\Hilm[F],u)\colon(A, \Hilm, J)\to(A_1, \Hilm_1, J_1)\)
  and
  \((\Hilm[F]_1,u_1)\colon(A_1, \Hilm_1, J_1)\to(A_2, \Hilm_2, J_2)\)
  be covariant correspondences and let \(\CP_{\Hilm[F],u}\)
  and \(\CP_{\Hilm[F]_1,u_1}\)
  be the associated correspondences of relative Cuntz--Pimsner
  algebras.  By definition,
  \(\CP_{\Hilm[F],u}\otimes_{\CP_{J_1,\Hilm_1}} \CP_{\Hilm[F]_1,u_1}\)
  and \(\CP_{\Hilm[F] \otimes_{A_1} \Hilm[F]_1,u\bullet u_1}\)
  are equal to
  \((\Hilm[F]\otimes_{A_1} \CP_{J_1,\Hilm[F]_1})
  \otimes_{\CP_{J_1,\Hilm[F]_1}} (\Hilm[F]_1 \otimes_{A_2}
  \CP_{J_2,\Hilm[F]_2})\)
  and
  \((\Hilm[F]\otimes_{A_1} \Hilm[F]_1) \otimes_{A_2}
  \CP_{J_2,\Hilm[F]_2}\)
  as \(\TT\)\nb-equivariant
  correspondences \(A \leadsto \CP_{J_2,\Hilm[F]_2}\).
  Associators and unit transformations give a canonical
  \(\TT\)\nb-equivariant
  isomorphism between these correspondences.  This isomorphism also
  intertwines the representations of~\(\Hilm\).
  Hence it is an isomorphism of correspondences
  \(\CP_{J,\Hilm[F]} \leadsto \CP_{J_2,\Hilm[F]_2}\).
  These canonical isomorphisms satisfy the coherence conditions for a
  homomorphism of bicategories in
  Definition~\ref{homomorphismbicategory}.
\end{proof}

The relative Cuntz--Pimsner algebra \(\CP_{J,\Hilm}\)
is the crossed product \(\CP^0_{J,\Hilm} \rtimes \CP^1_{J,\Hilm}\)
by Theorem~\ref{the:CP_Hilbi_crossed}.  So~\(\CP_{J,\Hilm}\)
with the gauge \(\TT\)\nb-action
and the Hilbert \(\CP^0_{J,\Hilm}\)-bimodule
\(\CP^1_{J,\Hilm}\)
contain the same amount of information.  We now study the construction
that sends \((A,\Hilm,J)\)
to the Hilbert \(\CP^0_{J,\Hilm}\)-bimodule
\(\CP^1_{J,\Hilm}\).
The appropriate bicategory of Hilbert bimodules is a sub-bicategory
of~\(\Corr^\NN_\proper\):

\begin{defn}
  \label{def:Corr_Bim}
  Let \(\Corr^\NN_\Bim \subseteq \Corr^\NN_\proper\)
  be the full sub-bicategory whose objects are triples
  \((B, \Hilm[G], I_{\Hilm[G]})\),
  where~\(\Hilm[G]\)
  is a Hilbert \(B\)\nb-bimodule
  and~\(I_{\Hilm[G]}\)
  is Katsura's ideal for~\(\Hilm[G]\),
  which is also equal to the range
  ideal~\(\BRAKET{\Hilm[G]}{\Hilm[G]}\)
  of the left inner product on~\(\Hilm[G]\).
  The arrows and \(2\)\nb-arrows
  among objects of~\(\Corr^\NN_\Bim\)
  are the same as in~\(\Corr^\NN_\proper\),
  including the condition
  \(I_{\Hilm}\Hilm[F]\subseteq \Hilm[F] I_{\Hilm[G]}\)
  for covariant correspondences.
\end{defn}

When we restrict the relative Cuntz--Pimsner algebra construction
\(\Corr^\NN_\proper \to \Corr^\TT\)
to~\(\Corr^\NN_\Bim\),
we get the (partial) crossed product construction for Hilbert bimodules
by Proposition~\ref{pro:Hilbi_CP}.  Thus Theorem~\ref{the:CP_functor}
also completes the crossed product for Hilbert bimodules to a functor
\(\Corr^\NN_\Bim \to \Corr^\TT\).

The map that sends \((A,\Hilm,J)\)
to \((\CP^0_{J,\Hilm},\CP^1_{J,\Hilm},I_{\CP^1_{J,\Hilm}})\)
is part of a functor \(\Corr^\NN_\proper \to \Corr^\NN_\Bim\)
which, when composed with the crossed product functor
\(\Corr^\NN_\Bim \to \Corr^\TT\),
gives the relative Cuntz--Pimsner algebra functor of
Theorem~\ref{the:CP_functor}.  We do not prove this now because it
follows from our main result.  The key step is the following universal
property of \((\CP^0_{J,\Hilm},\CP^1_{J,\Hilm},I_{\CP^1_{J,\Hilm}})\):

\begin{prop}
  \label{cptransformation}
  Let \((A,\Hilm,J)\)
  and \((B,\Hilm[G],I_{\Hilm[G]})\)
  be objects of \(\Corr^\NN_\proper\)
  and \(\Corr^\NN_\Bim\), respectively.  Let
  \[
  \upsilon_{(A,\Hilm,J)}\colon (A,\Hilm,J) \to (\CP^0_{J,\Hilm},\CP^1_{J,\Hilm},I_{\CP^1_{J,\Hilm}})
  \]
  be the covariant correspondence from
  Example~\textup{\ref{exa:universalarrow}}.  Composition
  with~\(\upsilon_{(A,\Hilm,J)}\) induces a groupoid equivalence
  \[
  \Corr^\NN_\proper\bigl((A,\Hilm,J), (B,\Hilm[G],I_{\Hilm[G]})\bigr)
  \simeq
  \Corr^\NN_\Bim\bigl((\CP^0_{J,\Hilm},\CP^1_{J,\Hilm},I_{\CP^1_{J,\Hilm}}), (B,\Hilm[G],I_{\Hilm[G]})\bigr).
  \]
\end{prop}

Recall that
\(\Corr^\NN_\proper\left((A,\Hilm, J), (A_1, \Hilm_1, J_1)\right)\)
for objects \((A,\Hilm, J)\)
and \((A_1, \Hilm_1, J_1)\)
of \(\Corr^\NN_\proper\)
denotes the groupoid with arrows
\((A, \Hilm, J)\to (A_1, \Hilm_1, J_1)\)
as objects and \(2\)\nb-arrows among them as arrows.

\begin{proof}
  We begin with an auxiliary construction.
  Proposition~\ref{pro:Hilbi_CP} identifies
  \(\CP_{I_{\Hilm[G]},\Hilm[G]} \cong B\rtimes\Hilm[G]\)
  as \(\ZZ\)\nb-graded
  \(\Cst\)\nb-algebras.
  In particular, \(\CP^0_{I_{\Hilm[G]},\Hilm[G]} \cong B\)
  and \(\CP^1_{I_{\Hilm[G]},\Hilm[G]} \cong \Hilm[G]\),
  \(\CP^{-1}_{I_{\Hilm[G]},\Hilm[G]} \cong \Hilm[G]^*\)
  as Hilbert \(B\)\nb-bimodules.
  Let \((\Hilm[F],u)\)
  be a proper covariant correspondence
  \((A,\Hilm,J)\to (B,\Hilm[G],I_{\Hilm[G]})\).
  It induces a proper, \(\TT\)\nb-equivariant
  correspondence \(\CP_{\Hilm[F],V} = \bigoplus_{n\in\ZZ} \CP_{\Hilm[F],V}^n\)
  from~\(\CP_{J,\Hilm}\)
  to~\(\CP_{I_{\Hilm[G]},\Hilm[G]}\)
  by Proposition~\ref{pro:CP_functorial}.  By construction,
  \(\CP_{\Hilm[F],V}^n = \Hilm[F] \otimes_B
  \CP^n_{I_{\Hilm[G]},\Hilm[G]}\).
  Thus
  \(\CP_{\Hilm[F],V}^0 = \Hilm[F] \otimes_B \CP^0_{I_{\Hilm[G]},\Hilm[G]}
  \cong \Hilm[F] \otimes_B B \cong \Hilm[F]\)
  and
  \(\CP_{\Hilm[F],V}^1 = \Hilm[F] \otimes_B \CP^1_{I_{\Hilm[G]},\Hilm[G]}
  \cong \Hilm[F] \otimes_B \Hilm[G]\).
  The left action on~\(\CP_{\Hilm[F],V}\)
  is a nondegenerate, \(\TT\)\nb-equivariant
  \Star{}homomorphism \(\CP_{J,\Hilm} \to \Comp(\CP_{\Hilm[F],V})\).
  So~\(\CP^0_{J,\Hilm}\)
  acts on~\(\CP_{\Hilm[F],V}\)
  by grading-preserving operators.  Restricting to the degree-\(0\)
  part, we get a nondegenerate \Star{}homomorphism
  \(\CP^0_{J,\Hilm} \to \Comp(\CP_{\Hilm[F],V}^0) \cong \Comp(\Hilm[F])\).
  Let~\(\Hilm[F]^\#\)
  be~\(\Hilm[F]\)
  viewed as a correspondence \(\CP^0_{J,\Hilm} \leadsto B\)
  in this way.

  We now construct an isomorphism of correspondences
  \[
  u^\#\colon \CP^1_{J,\Hilm} \otimes_{\CP^0_{J,\Hilm}} \Hilm[F]^\#
  \Rightarrow \Hilm[F]^\# \otimes_B \Hilm[G].
  \]
  We need two descriptions of~\(u^\#\).
  The first shows that it is unitary, the second that it intertwines
  the left actions of \(\CP^0_{J,\Hilm}\).
  The first formula for~\(u^\#\)
  uses Lemma~\ref{fibers}, which gives unitary Hilbert \(B\)\nb-module
  maps
  \[
  \CP^1_{J,\Hilm} \otimes_{\CP^0_{J,\Hilm}} \Hilm[F]^\#
  \cong \Hilm \otimes_A \CP^0_{J,\Hilm} \otimes_{\CP^0_{J,\Hilm}} \Hilm[F]^\#
  \cong \Hilm \otimes_A \Hilm[F].
  \]
  Composing with
  \(u\colon \Hilm \otimes_A \Hilm[F]\Rightarrow \Hilm[F] \otimes_B
  \Hilm[G]\)
  gives the desired unitary~\(u^\#\).
  The second formula for~\(u^\#\)
  restricts the left action of~\(\CP_{J,\Hilm}\)
  on~\(\CP_{\Hilm[F],V}\) to a multiplication map
  \begin{equation}
    \label{eq:other_multiplication_map}
    \CP^1_{J,\Hilm} \otimes_{\CP^0_{J,\Hilm}} \Hilm[F]^\#
    = \CP^1_{J,\Hilm} \otimes_{\CP^0_{J,\Hilm}} \CP_{\Hilm[F],V}^0
    \to \CP_{\Hilm[F],V}^1
    \cong \Hilm[F]^\#\otimes_B \Hilm[G].
  \end{equation}
  This is manifestly \(\CP^0_{J,\Hilm}\)\nb-linear
  because the isomorphism
  \(\Hilm[F]^\# \otimes_B \CP_{I_{\Hilm[G]},\Hilm[G]}^n \cong
  \CP_{\Hilm[F],V}^n\)
  is by right multiplication and so intertwines the left actions
  of~\(\CP_{J,\Hilm}^0\).
  The map in~\eqref{eq:other_multiplication_map} maps
  \(t_J(\xi)\otimes \eta\mapsto u(\xi\otimes\eta)\)
  for all \(\xi\in\Hilm\),
  \(\eta\in\Hilm[F]\).
  This determines it by Lemma~\ref{fibers}. So both constructions give
  the same map~\(u^\#\).

  We claim that
  \(I_{\CP^1_{J,\Hilm}} \cdot \Hilm[F]^\# \subseteq \Hilm[F]^\#\cdot
  I_{\Hilm[G]}\)
  holds, so that the pair \((\Hilm[F]^\#,u^\#)\)
  is a proper covariant correspondence from
  \((\CP^0_{J,\Hilm},\CP^1_{J,\Hilm},I_{\CP^1_{J,\Hilm}})\)
  to \((B,\Hilm[G],I_{\Hilm[G]})\).
  The ideal~\(I_{\CP^1_{J,\Hilm}}\)
  is equal to the range of the left inner product
  on~\(\CP^1_{J,\Hilm}\).
  Using the Fell bundle structure, we may rewrite this as
  \(\CP^1_{J,\Hilm}\cdot \CP^{-1}_{J,\Hilm}\).  Thus
  \[
  I_{\CP^1_{J,\Hilm}} \cdot \CP_{\Hilm[F],V}^0
  = \CP^1_{J,\Hilm}\cdot \CP^{-1}_{J,\Hilm} \cdot \CP_{\Hilm[F],V}^0
  \subseteq \CP^1_{J,\Hilm}\cdot \CP_{\Hilm[F],V}^{-1}
  = \Hilm\cdot \CP^0_{J,\Hilm}\cdot \CP_{\Hilm[F],V}^{-1}
  = \Hilm\cdot \CP_{\Hilm[F],V}^{-1}.
  \]
  The product \(\Hilm\cdot\CP_{\Hilm[F],V}^{-1}\)
  uses the representation of~\(\Hilm\)
  on~\(\CP_{\Hilm[F],V}\)
  built in the proof of Proposition~\ref{pro:CP_functorial}.  So
  \(\Hilm\cdot \CP_{\Hilm[F],V}^{-1}\) is the image of the map
  \[
  \Hilm \otimes_A \Hilm[F] \otimes_B \Hilm[G]^*
  \cong \Hilm[F] \otimes_B \Hilm[G] \otimes_B \Hilm[G]^*
  = \Hilm[F] \cdot I_{\Hilm[G]}.
  \]
  So
  \(I_{\CP^1_{J,\Hilm}} \cdot \CP_{\Hilm[F],V}^0 \subseteq \Hilm[F] \cdot
  I_{\Hilm[G]}\)
  as claimed.  We have turned a proper covariant correspondence
  \((\Hilm[F],u)\)
  from \((A,\Hilm,J)\)
  to \((B,\Hilm[G],I_{\Hilm[G]})\)
  into a proper covariant correspondence \((\Hilm[F]^\#,u^\#)\)
  from \((\CP^0_{J,\Hilm},\CP^1_{J,\Hilm},I_{\CP^1_{J,\Hilm}})\)
  to \((B,\Hilm[G],I_{\Hilm[G]})\).

  Conversely, take a proper covariant correspondence \((\Hilm[F],u)\)
  from \((\CP^0_{J,\Hilm},\CP^1_{J,\Hilm},I_{\CP^1_{J,\Hilm}})\)
  to \((B,\Hilm[G],I_{\Hilm[G]})\).
  Composing it with~\(\upsilon_{(A,\Hilm,J)}\)
  gives a proper covariant correspondence from \((A,\Hilm,J)\)
  to \((B,\Hilm[G],I_{\Hilm[G]})\).
  We now simplify this product of covariant correspondences.  The
  underlying correspondence \(A\to \CP^0_{J,\Hilm}\)
  in~\(\upsilon_{(A,\Hilm,J)}\)
  is \(\CP^0_{J,\Hilm}\),
  and the isomorphism
  \(\Hilm \otimes_A \CP^0_{J,\Hilm} \cong \CP^0_{J,\Hilm}
  \otimes_{\CP^0_{J,\Hilm}} \CP^1_{J,\Hilm} = \CP^1_{J,\Hilm}\)
  is the one from Lemma~\ref{fibers}.  We identify the tensor product
  \(\CP^0_{J,\Hilm} \otimes_{\CP^0_{J,\Hilm}} \Hilm[F]\)
  with~\(\Hilm[F]\)
  by the canonical map.  Thus the product of \((\Hilm[F],u)\)
  with~\(\upsilon_{(A,\Hilm,J)}\)
  is canonically isomorphic to a covariant correspondence
  \((\Hilm[F]^\flat,u^\flat)\)
  with underlying correspondence
  \(\Hilm[F]^\flat = \Hilm[F]\colon A\leadsto B\)
  with the left \(A\)\nb-action
  through \(\pi_J\colon A\to \CP^0_{J,\Hilm}\).
  The isomorphism
  \(u^\flat \colon \Hilm \otimes_A \Hilm[F]^\flat \Rightarrow
  \Hilm[F]^\flat \otimes_B \Hilm[G]\)
  is the composite of the given isomorphism
  \(u\colon \CP^1_{J,\Hilm} \otimes_{\CP^u_{J,\Hilm}} \Hilm[F]
  \Rightarrow \Hilm[F] \otimes_B \Hilm[G]\)
  with the isomorphism
  \(\Hilm \otimes_A \CP^0_{J,\Hilm} \cong \CP^1_{J,\Hilm}\)
  from Lemma~\ref{fibers}.

  Now let \((\Hilm[F],u)\)
  be a proper covariant correspondence from \((A,\Hilm,J)\)
  to \((B,\Hilm[G],I_{\Hilm[G]})\).  We claim that
  \begin{equation}
    \label{sharp-flat}
    (\Hilm[F]^{\#\flat},u^{\#\flat}) = (\Hilm[F],u).
  \end{equation}
  By construction, the underlying Hilbert \(B\)\nb-module
  of~\(\Hilm[F]^{\#\flat}\)
  is~\(\Hilm[F]\).
  We even have \(\Hilm[F]^{\#\flat}=\Hilm[F]\)
  as correspondences \(A\leadsto B\),
  that is, the left \(\CP^0_{J,\Hilm}\)\nb-action
  on~\(\Hilm[F]^\#\)
  composed with \(\pi_J\colon A\to\CP^0_{J,\Hilm}\)
  is the original action of~\(A\).
  The isomorphism
  \(\Hilm \otimes_A \CP^0_{J,\Hilm} \cong \CP^1_{J,\Hilm}\)
  is used both to get~\(u^\#\)
  from~\(u\)
  and to get~\(u^{\#\flat}\)
  from~\(u^\#\).  Unravelling this shows that \(u^{\#\flat} = u\).

  Now we claim that the map that sends a proper covariant
  correspondence
  \((\Hilm[F],u)\colon
  (\CP^0_{J,\Hilm},\CP^1_{J,\Hilm},I_{\CP^1_{J,\Hilm}}) \to
  (B,\Hilm[G],I_{\Hilm[G]})\)
  to \((\Hilm[F]^\flat,u^\flat)\)
  is injective.  This claim and~\eqref{sharp-flat} imply
  \((\Hilm[F]^{\flat\#},u^{\flat\#}) = (\Hilm[F],u)\),
  that is, our two operations are inverse to each other.  To prove
  injectivity, we use Proposition~\ref{pro:CP_functorial} to build a
  correspondence
  \(\CP_{\Hilm[F],u}\colon \CP_{J,\Hilm}\leadsto
  \CP_{I_{\Hilm[G]},\Hilm[G]}\)
  from~\((\Hilm[F],u)\).
  This correspondence determines~\((\Hilm[F],u)\):
  we can get back~\(\Hilm[F]\)
  as its degree-\(0\)
  part because \(\CP_{I_{\Hilm[G]},\Hilm[G]} = B\rtimes \Hilm[G]\),
  and because \(u\)
  and the left \(\CP^0_{J,\Hilm}\)\nb-module
  structure on~\(\Hilm[F]\)
  are both contained in the left \(\CP_{J,\Hilm}\)\nb-module
  structure on~\(\CP_{\Hilm[F],u}\).
  An \(\CP_{J,\Hilm}\)\nb-module
  structure on \(\CP_{I_{\Hilm[G]},\Hilm[G]}\)
  is already determined by a representation of~\((A,\Hilm)\).
  Since
  \(\CP_{I_{\Hilm[G]},\Hilm[G]}^n = \CP_{I_{\Hilm[G]},\Hilm[G]}^0\cdot
  \CP_{I_{\Hilm[G]},\Hilm[G]}^n\),
  this representation is determined by its restriction to
  \(\CP_{I_{\Hilm[G]},\Hilm[G]}^0 \cong B\).
  And~\((\Hilm[F]^\flat,u^\flat)\)
  determines the representation of~\((A,\Hilm)\)
  on~\(B\).
  Thus~\((\Hilm[F]^\flat,u^\flat)\) determines~\((\Hilm[F],u)\).

  The constructions of \((\Hilm[F]^\#,u^\#)\)
  and~\((\Hilm[F]^\flat,u^\flat)\)
  are clearly natural for isomorphisms of covariant
  correspondences.  So they form an isomorphism of categories
  \[
  \Corr^\NN_\proper\bigl((A,\Hilm,J), (B,\Hilm[G],I_{\Hilm[G]})\bigr)
  \cong
  \Corr^\NN_\Bim\bigl((\CP^0_{J,\Hilm},\CP^1_{J,\Hilm},I_{\CP^1_{J,\Hilm}}), (B,\Hilm[G])\bigr).
  \]
  One piece in this isomorphism is naturally equivalent to the functor
  that composes with~\(\upsilon_{(A,\Hilm,J)}\).
  Hence this functor is an equivalence of categories, as asserted.
\end{proof}

\section{The reflector from correspondences to Hilbert bimodules}
\label{mainsection1}

We now strengthen Proposition~\ref{cptransformation} using some
general results on adjunctions of homomorphisms between bicategories.
We first recall the related and better known results about ordinary
categories and functors.

Let \(\mathcal{C}\)
and~\(\mathcal{B}\)
be categories.  Let \(R\colon \mathcal{C}\to\mathcal{B}\)
be a functor and \(b\in \obj \mathcal{B}\).
An object \(c\in\obj \mathcal{C}\)
with an arrow \(\upsilon\colon b\to R(c)\)
is called a \emph{universal arrow} from~\(b\)
to~\(R\)
if, for each \(x\in \obj \mathcal{C}\)
and each \(f\in\mathcal{B}(b, R(x))\),
there is a unique \(g\in\mathcal{C}(c,x)\)
with \(R(g)\circ \upsilon = f\).  Equivalently, the maps
\begin{equation}
  \label{adjunction_category}
  \mathcal{C}(c,x) \to \mathcal{B}(b,R(x)),\qquad
  f\mapsto R(f)\circ \upsilon,
\end{equation}
are bijective for all \(x\in \obj \mathcal{C}\).
The functor~\(R\)
has a left adjoint \(L\colon \mathcal{B}\to\mathcal{C}\)
if and only if such universal arrows exist for all
\(x\in\obj \mathcal{C}\).
The left adjoint functor \(L\colon \mathcal{B}\to\mathcal{C}\)
is uniquely determined up to natural isomorphism.  It maps
\(b\mapsto c\)
on objects, and the isomorphisms~\eqref{adjunction_category} become
natural in both \(b\)
and~\(x\)
when we replace~\(c\)
by~\(L(b)\).
An adjunction between \(L\)
and~\(R\)
may also be expressed through its unit and counit, that is, natural
transformations \(L\circ R \Rightarrow \id_\mathcal{C}\)
and \(\id_\mathcal{B}\Rightarrow R\circ L\)
such that the induced transformations
\(L\Rightarrow L\circ R\circ L \Rightarrow L\)
and \(R\Rightarrow R\circ L\circ R \Rightarrow R\)
are unit transformations.

A subcategory \(\mathcal{C}\subseteq\mathcal{B}\)
is called \emph{reflective} if the inclusion functor
\(R\colon \mathcal{C}\to\mathcal{B}\)
has a left adjoint \(L\colon \mathcal{B}\to\mathcal{C}\).
The functor~\(L\)
is called \emph{reflector}.  The case we care about is a bicategorical
version of a full subcategory.  If \(\mathcal{C}\subseteq\mathcal{B}\)
is a full subcategory, then we may choose \(L\circ R\)
to be the identity functor on~\(\mathcal{C}\)
and the counit \(L\circ R\Rightarrow \id_\mathcal{C}\)
to be the unit natural transformation.

Fiore~\cite{Fiore:Pseudo_biadjoints} carries the story of adjoint
functors over to homomorphisms between \(2\)\nb-categories
(which he calls ``pseudo functors''), that is, bicategories where the
associators and unitors are identity \(2\)\nb-arrows.
The bicategories we need are not \(2\)\nb-categories.
But any bicategory is equivalent to a \(2\)\nb-category
by MacLane's Coherence Theorem.  Hence Fiore's definitions and results
apply in bicategories as well.
We shorten notation by speaking of ``universal'' arrows and
``adjunctions'' instead of ``biuniversal'' arrows and
``biadjunctions.''  A \(2\)\nb-category
is also a category with some extra structure.  So leaving out the
prefix~``bi'' may cause confusion in that setting.  But it will always
be clear whether we mean the categorical or bicategorical notions.

\begin{defn}[\cite{Fiore:Pseudo_biadjoints}*{Definition 9.4}]
  \label{def:universalarrow}
  Let \(\mathcal{B}\)
  and~\(\mathcal{C}\)
  be bicategories, \(R\colon\mathcal{C}\to\mathcal{B}\)
  a homomorphism, and \(b\in\obj\mathcal{B}\).
  Let \(c\in\obj\mathcal{C}\)
  and let \(g\colon b\to R(c)\)
  be an arrow in~\(\mathcal{B}\).
  The pair \((c, g)\)
  is a \emph{universal arrow} from~\(b\)
  to~\(R\)
  if, for every \(x\in\obj\mathcal{C}\),
  the following functor is an equivalence of categories:
  \[
  g^*\colon \mathcal{C}(c,x)\to\mathcal{B}(b,R(x)),\qquad
  f\mapsto R(f)\cdot g,\quad w\mapsto R(w)\bullet 1_g.
  \]
\end{defn}

Universal arrows are called \emph{left biliftings} by
Street~\cite{Street:Fibrations_bicategories}.

We can now reformulate Proposition~\ref{cptransformation}:

\begin{prop}
  \label{pro:Corr_universal_arrow}
  Let \((A,\Hilm,J) \in \obj \Corr^\NN_\proper\).
  The covariant correspondence \(\upsilon_{(A,\Hilm,J)}\)
  from \((A,\Hilm,J)\)
  to \((\CP^0_{J,\Hilm},\CP^1_{J,\Hilm},I_{\Hilm})\)
  is a universal arrow from \((A,\Hilm,J)\)
  to the inclusion homomorphism
  \(\Corr^\NN_\Bim \to \Corr^\NN_\proper\).\qed
\end{prop}

There are two alternative definitions of adjunctions, based on
equivalences between morphism categories or on units and counits.
These are spelled out, respectively, by Fiore in
\cite{Fiore:Pseudo_biadjoints}*{Definition 9.8} and by Gurski in
\cite{Gurski:Biequivalences_tricategories}*{Definition 2.1}.  We shall
use Fiore's definition.

\begin{defn}[\cite{Fiore:Pseudo_biadjoints}*{Definition 9.8}]
  \label{def:adjunction}
  Let \(\mathcal{B}\)
  and~\(\mathcal{C}\)
  be bicategories.  An \emph{adjunction} between them consists of
  \begin{itemize}
  \item two homomorphisms \(L\colon \mathcal{B}\to\mathcal{C}\),
    \(R\colon \mathcal{C}\to\mathcal{B}\);
  \item equivalences of categories
    \[
    \varphi_{b,c}\colon \mathcal{C}(L(b),c) \simeq \mathcal{B}(b,R(c))
    \]
    for all \(b\in\obj \mathcal{B}\), \(c\in\obj \mathcal{C}\);
  \item natural equivalences of functors
    \[
    \xymatrix{
      \mathcal{C}(L(b_1),c_1) \ar[r]^{f^*} \ar[d]_{\varphi_{b_1,c_1}} &
      \mathcal{C}(L(b_2),c_1) \ar[r]^{g_*}  &
      \mathcal{C}(L(b_2),c_2) \ar[d]^{\varphi_{b_2,c_2}} \ar@{=>}[dll] \\
      \mathcal{B}(b_1,R(c_1)) \ar[r]_{f^*}&
      \mathcal{B}(b_2,R(c_1)) \ar[r]_{g_*}&
      \mathcal{B}(b_2,R(c_2))
    }
    \]
    for all arrows \(f\colon b_2 \to b_1\),
    \(g\colon c_1 \to c_2\) in \(\mathcal{B}\) and~\(\mathcal{C}\).
  \end{itemize}
  These are subject to a coherence condition.  In brief, the
  functors~\(\varphi_{b,c}\)
  and the natural equivalences form a transformation between the
  homomorphisms
  \[
  \mathcal{B}^\op \times \mathcal{C}\rightrightarrows
  \mathbf{Cat},\qquad
  (b,c)\mapsto \mathcal{C}(L(b),c),\ \mathcal{B}(b,R(c)).
  \]
\end{defn}

Here~\(\mathbf{Cat}\)
is the bicategory of categories, see Example~\ref{exa:Cat}.

\begin{thm}[\cite{Fiore:Pseudo_biadjoints}*{Theorem 9.17}]
  \label{the:adjoint_criterion}
  Let \(\mathcal{B}\)
  and~\(\mathcal{C}\)
  be bicategories and let \(R\colon \mathcal{C}\to\mathcal{B}\)
  be a homomorphism.  It is part of an adjunction if and only if
  there are universal arrows from~\(c\)
  to~\(R\) for each object \(c\in \obj \mathcal{C}\).
\end{thm}

More precisely, let \(c_b\in\obj\mathcal{C}\)
and \(\upsilon_b\colon b\to R(c_b)\)
for \(b\in\obj\mathcal{C}\)
be universal arrows from~\(b\)
to~\(R\).
Then there is an adjoint homomorphism
\(L\colon \mathcal{B}\to\mathcal{C}\)
that maps \(b\mapsto c_b\)
on objects.  In particular, this assignment is part of a homomorphism
of bicategories.

\begin{thm}[\cite{Fiore:Pseudo_biadjoints}*{Theorem~9.20}]
  \label{uniquenessofadjunction}
  Two left adjoints
  \(L,L'\colon\mathcal{B}\rightrightarrows\mathcal{C}\)
  of \(R\colon\mathcal{C}\to\mathcal{B}\)
  are equivalent, that is, there are transformations
  \(L\Rightarrow L'\)
  and \(L'\Rightarrow L\)
  that are inverse to each other up to invertible modifications.
\end{thm}

Using these general theorems, we may strengthen
Proposition~\ref{cptransformation} (in the form of
Proposition~\ref{pro:Corr_universal_arrow}) to an adjunction theorem:

\begin{cor}
  \label{cor:reflective}
  The sub-bicategory \(\Corr^\NN_\Bim \subseteq \Corr^\NN_\proper\)
  is reflective, that is, the inclusion homomorphism
  \(R\colon \Corr^\NN_\Bim \to \Corr^\NN_\proper\)
  has a left adjoint \textup{(}reflector\textup{)}
  \(L\colon \Corr^\NN_\proper \to \Corr^\NN_\Bim\).
  On objects, this adjoint homomorphism maps
  \[
  (A,\Hilm,J) \mapsto (\CP^0_{J,\Hilm},\CP^1_{J,\Hilm},I_{\CP^1_{J,\Hilm}}).
  \]
\end{cor}

The homomorphism~\(L\)
is determined uniquely up to equivalence by
Theorem~\ref{uniquenessofadjunction}.  So we have
characterised the construction of relative Cuntz--Pimsner algebras in
bicategorical terms, as the reflector for the full sub-bicategory
\(\Corr^\NN_\Bim \subseteq \Corr^\NN_\proper\).
By Corollary~\ref{cor:reflective}, the relative Cuntz--Pimsner algebra
construction is part of a homomorphism
\(L\colon \Corr^\NN_\proper \to \Corr^\NN_\Bim\).
For instance, this implies the following:

\begin{cor}
  The relative Cuntz--Pimsner algebras \(\CP_{J,\Hilm}\)
  and \(\CP_{J_1,\Hilm_1}\)
  are Morita equivalent if there is a Morita equivalence~\(\Hilm[F]\)
  between \(\Hilm\)
  and~\(\Hilm_1\)
  as in
  \cite{Muhly-Solel:Morita_equivalence_of_tensor_algebras}*{Definition~\textup{2.1}}
  with \(J\cdot \Hilm[F]\subseteq\Hilm[F] \cdot J_1\).
\end{cor}

The proof of Theorem~\ref{the:adjoint_criterion} also describes the
adjoint functor.  We now describe the reflector
\(L\colon \Corr^\NN_\proper \to \Corr^\NN_\Bim\)
explicitly, thereby explaining part of the proof of
Theorem~\ref{the:adjoint_criterion}.  Much of the work in this proof
is needed to check that various diagrams of \(2\)\nb-arrows
commute.  We do not repeat these computations here.

The homomorphism~\(L\)
maps
\((A, \Hilm, J)\mapsto
(\CP^0_{J,\Hilm},\CP^1_{J,\Hilm},I_{\CP^1_{J,\Hilm}})\)
on objects.  Let \((A, \Hilm, J)\)
and \((A_1, \Hilm_1,J_1)\)
be objects of \(\Corr^\NN_\proper\)
and let \((\Hilm[F], u)\colon(A, \Hilm, J)\to(A_1, \Hilm_1,J_1)\)
be proper covariant correspondences.  We use the notation of the proof
of Proposition~\ref{cptransformation} and
write~\(\bar{\iota}_{\Hilm_1}\)
for the canonical isomorphism
\(\Hilm_1\otimes_{A_1} \CP^0_{J_1,\Hilm_1} \cong
\CP^1_{J_1,\Hilm_1}\otimes_{\CP^0_{J_1,\Hilm_1}} \CP^0_{J_1,\Hilm_1}\)
from Lemma~\ref{fibers}, which is the covariance part
of~\(\upsilon_{(A_1,\Hilm_1,J_1)}\).  Let
\begin{multline*}
  L(\Hilm[F], u)\colon
  (\CP^0_{J,\Hilm},\CP^1_{J,\Hilm},I_{\CP^1_{J,\Hilm}}) \to
  (\CP^0_{J_1,\Hilm_1},\CP^1_{J_1,\Hilm_1},I_{\CP^1_{J_1,\Hilm_1}}),\\
  L(\Hilm[F], u) \defeq
  \left((\Hilm[F]\otimes_{A_1}\CP^0_{J_1, \Hilm_1})^\#, 
    (u\bullet\bar{\iota}_{\Hilm_1})^\#\right).
\end{multline*}
In other words, we first compose~\((\Hilm[F], u)\)
with~\(\upsilon_{(A_1,\Hilm_1,J_1)}\)
to get a covariant correspondence
\((\Hilm[F]\otimes_{A_1}\CP^0_{J_1, \Hilm_1},
u\bullet\bar{\iota}_{\Hilm_1})\)
from \((A, \Hilm, J)\)
to
\((\CP^0_{J_1,\Hilm_1},\CP^1_{J_1,\Hilm_1},I_{\CP^1_{J_1,\Hilm_1}})\)
and then apply the equivalence in Proposition~\ref{cptransformation}.

The construction on covariant correspondences above is clearly
``natural'', that is, functorial for isomorphisms.  Explicitly, \(L\)
maps an isomorphism of covariant correspondences
\(w\colon(\Hilm[F],u)\Rightarrow (\Hilm[F]', u')\) to
\[
L(w)\defeq (w\otimes 1_{\CP^0_{J_1, \Hilm_1}})^\#\colon
L(\Hilm[F],u)\Rightarrow L(\Hilm[F]',u').
\]

To make~\(L\)
a homomorphism, we also need compatibility data for units and
composition of arrows.  The construction of~\(L\)
above maps the identity covariant correspondence on~\((A,\Hilm,J)\)
to
\(\upsilon_{(A,\Hilm,J)}^\#\colon
(\CP^0_{J,\Hilm},\CP^1_{J,\Hilm},I_{\Hilm}) \to
(\CP^0_{J,\Hilm},\CP^1_{J,\Hilm},I_{\Hilm})\).
This is canonically isomorphic to the identity covariant
correspondence on~\((\CP^0_{J,\Hilm},\CP^1_{J,\Hilm},I_{\Hilm})\)
because the equivalence in Proposition~\ref{cptransformation} is by
composition with~\(\upsilon_{(A,\Hilm,J)}\).
This is the unit part in our homomorphism~\(L\).

Let \((\Hilm[F], u)\colon(A, \Hilm, J)\to(A_1, \Hilm_1,J_1)\)
and
\((\Hilm[F]_1, u_1)\colon(A_1, \Hilm_1, J_1)\to (A_2, \Hilm_2,J_2)\)
be proper covariant correspondences.  Then the homomorphism~\(L\)
contains isomorphisms of covariant correspondences
\begin{equation}
  \label{eq:lambda_1}
  \lambda\bigl((\Hilm[F], u), (\Hilm[F]_1, u_1)\bigr)\colon
  L(\Hilm[F], u)\circ L (\Hilm[F]_1, u_1) \Rightarrow
  L\bigl((\Hilm[F], u)\circ (\Hilm[F]_1, u_1)\bigr),
\end{equation}
which are natural for isomorphisms of covariant correspondences and
satisfy some coherence conditions when we compose three covariant
correspondences or compose with identity covariant correspondences.
We take~\(\lambda\) to be the isomorphism
\[
(\Hilm[F]_0\otimes_{A_1}\CP^0_{J_1,\Hilm_1})\otimes_{\CP^0_{J_1,\Hilm_1}}
(\Hilm[F]_1\otimes_{A_2}\CP^0_{J_2,\Hilm_2})\cong
(\Hilm[F]_0\otimes_{A_1}\Hilm[F]_1)\otimes_{A_2} \CP^0_{J_2,\Hilm_2}
\]
given by the left action of~\(\CP^0_{J_1,\Hilm_1}\)
on \((\Hilm[F]_1\otimes_{A_2}\CP^0_{J_2,\Hilm_2})\)
that is constructed in the proof of
Proposition~\ref{pro:CP_functorial}.

The proof of Theorem~\ref{the:adjoint_criterion} builds~\(\lambda\)
using only the universality of the arrows~\(\upsilon_{(A,\Hilm,J)}\).
By the equivalence of categories in
Proposition~\ref{cptransformation}, whiskering (horizontal
composition) with~\(\upsilon_{(A,\Hilm,J)}\)
maps isomorphisms as in~\eqref{eq:lambda_1} bijectively to
isomorphisms
\begin{equation}
  \label{eq:lambda_2}
  \upsilon_{(A,\Hilm,J)} \circ L(\Hilm[F], u)\circ L (\Hilm[F]_1, u_1) \Rightarrow
  \upsilon_{(A,\Hilm,J)} \circ L\bigl((\Hilm[F], u)\circ (\Hilm[F]_1, u_1)\bigr).
\end{equation}
The construction of~\(L\)
implies
\(\upsilon_{(A,\Hilm,J)} \circ L(\Hilm[F], u)\circ L (\Hilm[F]_1, u_1)
\cong (\Hilm[F], u)\circ \upsilon_{(A_1,\Hilm_1,J_1)} \circ L
(\Hilm[F]_1, u_1) \cong (\Hilm[F], u)\circ (\Hilm[F]_1, u_1) \circ
\upsilon_{(A_2,\Hilm_2,J_2)}\)
and
\(\upsilon_{(A,\Hilm,J)} \circ L\bigl((\Hilm[F], u)\circ (\Hilm[F]_1,
u_1)\bigr) \cong \bigl((\Hilm[F], u)\circ (\Hilm[F]_1, u_1)\bigr)
\circ \upsilon_{(A_2,\Hilm_2,J_2)}\),
where we disregard associators.  Hence there is a canonical
isomorphism of covariant correspondences as in~\eqref{eq:lambda_2}.
This Ansatz produces the same isomorphisms~\(\lambda\)
as above.  We have now described the data of the homomorphism~\(L\).
Fiore's arguments in~\cite{Fiore:Pseudo_biadjoints} show that it is
indeed a homomorphism.

\begin{prop}
  The composite of~\(L\)
  and the crossed product homomorphism
  \(\Corr^\NN_\Bim \to \Corr^\TT\)
  is naturally isomorphic to the homomorphism
  \(\Corr^\NN_\proper \to \Corr^\TT\)
  of Theorem~\textup{\ref{the:CP_functor}}.
\end{prop}

\begin{proof}
  Our homomorphisms agree on objects by
  Proposition~\ref{katsurasalgebra}.  The proof of
  Proposition~\ref{cptransformation} constructed the covariant
  correspondence \((\Hilm[F]^\#,u^\#)\)
  by taking the degree-\(0\)
  part in the correspondence constructed in the proof of
  Proposition~\ref{pro:CP_functorial}.  Thus we may build a natural
  isomorphism between the functors in question out of the
  nondegenerate left action of~\(\CP^0_{J_1,\Hilm_1}\)
  on~\(\CP_{J_1,\Hilm_1}\).
\end{proof}

So the reflector~\(L\)
lifts the Cuntz--Pimsner algebra homomorphism
\(\Corr^\NN_\proper \to \Corr^\TT\)
to a homomorphism with values in~\(\Corr^\NN_\Bim\).
Such a lifting should exist because a Hilbert
bimodule and its crossed product with the \(\TT\)\nb-action
determine each other.

An adjunction also contains ``natural'' equivalences of categories
\(\varphi_{b,c}\colon \mathcal{C}(L(b),c) \simeq
\mathcal{B}(b,R(c))\),
where naturality is further data, see Definition~\ref{def:adjunction}.
In the case at hand, these equivalences are exactly the equivalences
of categories
\[
\upsilon_{(A,\Hilm,J)}^*\colon  \Corr^\NN_\proper\bigl((A,\Hilm,J), (B,\Hilm[G],I_{\Hilm[G]})\bigr)
\simeq
\Corr^\NN_\Bim\bigl((\CP^0_{J,\Hilm},\CP^1_{J,\Hilm},I_{\CP^1_{J,\Hilm}}), (B,\Hilm[G],I_{\Hilm[G]})\bigr).
\]
in Proposition~\ref{cptransformation}.  Their naturality boils down to
the canonical isomorphisms of correspondences
\(\upsilon_{(A,\Hilm,J)} \circ L(\Hilm[F], u) \cong (\Hilm[F], u)
\circ \upsilon_{(A,\Hilm,J)}\),
which we have already used above to describe the multiplicativity
data~\(\lambda\) in the homomorphism~\(L\).

Finally, we relate our adjunction to the colimit description of
Cuntz--Pimsner algebras in~\cite{Albandik-Meyer:Colimits}.  Let
\(\mathcal{C}\) and~\(\mathcal{D}\) be categories.
Let~\(\mathcal{C}^\mathcal{D}\)
be the category of functors \(\mathcal{D}\to\mathcal{C}\),
which are also called diagrams of shape~\(\mathcal{D}\)
in~\(\mathcal{C}\).
Identify~\(\mathcal{C}\)
with the subcategory of ``constant'' diagrams
in~\(\mathcal{C}^\mathcal{D}\).
This subcategory is reflective if and only if all
\(\mathcal{D}\)\nb-shaped
diagrams in~\(\mathcal{C}\)
have a colimit, and the reflector maps a diagram to its colimit.

This remains true for the bicategorical colimits
in~\cite{Albandik-Meyer:Colimits}: by definition, the colimit of a
diagram is a universal arrow to a constant diagram.  In our context, a
constant diagram in~\(\Corr^\NN_\proper\)
is an object of the form \((B,B,B)\)
that is, the Hilbert \(B\)\nb-bimodule
is the identity bimodule and \(J=B\)
as always for objects of~\(\Corr^\NN_\Bim\).
Since the condition \(J\cdot\Hilm[F] \subseteq \Hilm[F]\cdot B\)
always holds, the ideal~\(J\)
plays no role, compare Remark~\ref{rem:covariance_simplifies_Bmax}.

A proper covariant correspondence \((A,\Hilm,J) \to (B,B,B)\)
is equivalent to a proper correspondence
\(\Hilm[F]\colon A\leadsto B\)
with an isomorphism \(\Hilm \otimes_A \Hilm[F] \Rightarrow \Hilm[F]\)
because \(\Hilm[F]\otimes_B B \cong \Hilm[F]\).
As shown in~\cite{Albandik-Meyer:Colimits}, such a pair is equivalent
to a representation \((\varphi,t)\)
of the correspondence~\(\Hilm\)
on~\(\Hilm[F]\)
that is nondegenerate in the sense that
\(t(\Hilm)\cdot \Hilm[F]=\Hilm[F]\).
The properness of~\(\Hilm[F]\)
means that \(\varphi(A)\subseteq \Comp(\Hilm[F])\),
which implies \(t(\Hilm) \subseteq\Comp(\Hilm[F])\).

It is shown in~\cite{Albandik-Meyer:Colimits} that all diagrams of
\emph{proper} correspondences of any shape have a colimit.  This is
probably false for diagrams of non-proper correspondences, such as the
correspondence \(\ell^2(\NN)\colon \CC\leadsto \CC\)
that defines the Cuntz algebra~\(\CP_\infty\).
The way around this problem that we found here is to enlarge the
sub-bicategory of constant diagrams, allowing diagrams of Hilbert
bimodules.  In addition, we added an ideal~\(J\)
to have enough data to build \emph{relative} Cuntz--Pimsner
algebras.

Since the sub-bicategory \(\Corr\subseteq \Corr^\NN_\proper\)
of constant diagrams is contained in~\(\Corr^\NN_\Bim\),
we may relate universal arrows to objects in~\(\Corr\)
and \(\Corr^\NN_\Bim\)
as follows.  Let \((A,\Hilm,J)\)
be an object of \(\Corr^\NN_\proper\).  Then
\(\upsilon_{(A,\Hilm,J)} \colon (A,\Hilm,J) \to
(\CP^0_{J,\Hilm},\CP^1_{J,\Hilm},I_{\CP^1_{J,\Hilm}})\)
is a universal arrow to an object of~\(\Corr^\NN_\Bim\)
by Proposition~\ref{pro:Corr_universal_arrow}.  The universality
of~\(\upsilon_{(A,\Hilm,J)}\)
implies that a universal arrow from~\((A,\Hilm,J)\)
to a constant diagram factors through~\(\upsilon_{(A,\Hilm,J)}\),
and that an arrow from
\((\CP^0_{J,\Hilm},\CP^1_{J,\Hilm},I_{\CP^1_{J,\Hilm}})\)
to a constant diagram is universal if and only if its composite
with~\(\upsilon_{(A,\Hilm,J)}\)
is universal.  In other words, the diagram \((A,\Hilm,J)\)
has a colimit if and only if
\((\CP^0_{J,\Hilm},\CP^1_{J,\Hilm},I_{\CP^1_{J,\Hilm}})\)
has one, and then the two colimits are the same.  We are dealing with
the same colimits as in~\cite{Albandik-Meyer:Colimits} because the
ideal~\(J\)
in \((A,\Hilm,J)\) plays no role for arrows to constant diagrams.

\appendix
\section{Bicategories}

We recall some basic definitions from bicategory theory, following
\cites{Benabou:Bicategories, Gray:lns}.  We also give a few examples
with Sections \ref{mainsection} and~\ref{mainsection1} in mind.

\begin{defn}
  \label{def:bicategory}
  A \emph{bicategory}~\(\mathcal{B}\) consists of the following data:
  \begin{itemize}
  \item a set of objects \(\obj\mathcal{B}\);
  \item a category \(\mathcal{B}(x,y)\)
    for each pair of objects \((x,y)\);
    objects of \(\mathcal{B}(x,y)\)
    are called \emph{arrows} (or \emph{morphisms}) from~\(x\)
    to~\(y\),
    and arrows in \(\mathcal{B}(x,y)\)
    are called \emph{\(2\)\nb-arrows}
    (or \emph{\(2\)\nb-morphisms});
    the category structure on \(\mathcal{B}(x,y)\)
    gives us a \emph{unit \(2\)\nb-arrow}~\(1_f\)
    on each arrow \(f\colon x\to y\),
    and a \emph{vertical composition} of \(2\)\nb-arrows:
    \(w_0\colon f_0\Rightarrow f_1\)
    and \(w_1\colon f_1\Rightarrow f_2\)
    compose to a \(2\)\nb-arrow
    \(w_1\cdot w_0\colon f_0\Rightarrow f_2\);

  \item composition functors
    \[
    {\circ}\colon \mathcal{B}(y,z)\times\mathcal{B}(x,y)\to\mathcal{B}(x,z)
    \]
    for each triple of objects \((x,y,z)\);
    this contains a \emph{horizontal composition} of \(2\)\nb-arrows
    as displayed below:
    \[
    \xymatrix{
      x\ar@/^1.2pc/[rr]^{f_0}="a" \ar@/^-1.2pc/[rr]_{f_1}="b"& &
      y\ar@/^1.2pc/[rr]^{g_0}="c"\ar@/^-1.2pc/[rr]_{g_1}="d" & &
      z \ar@{=>}^{w_0}"a";"b" \ar@{=>}^{w_1}"c";"d" 
    }\mapsto
    \xymatrix{
      x\ar@/^1.2pc/[rrrr]^{g_0\cdot f_0}="a" \ar@/^-1.2pc/[rrrr]_{g_1\cdot f_1}="b" & & &&
      z. \ar@{=>}^{w_1\bullet w_0 }"a";"b"
    }
    \]

  \item a unit arrow \(1_x\in\mathcal{B}(x,x)\) for each~\(x\);

  \item natural invertible \(2\)\nb-arrows
    (\emph{unitors})
    \(r_f\colon f\cdot 1_x \Rightarrow f\)
    and \(l_f\colon 1_y\cdot f \Rightarrow f\)
    for all \(f\in\mathcal{B}(x,y)\);

  \item natural isomorphisms
    \[
    \xymatrix{
      \mathcal{B}(x,y)\times\mathcal{B}(y,z)\times\mathcal{B}(z,w)
      \ar[rr]^{(\circ,1)} \ar[d]_{(1,\circ)}& &
      \mathcal{B}(x,z)\times\mathcal{B}(z,w)\ar[d]^{\circ} \\
      \mathcal{B}(x,y)\times\mathcal{B}(y,w) \ar[rr]^{\circ}\ar@{=>}[rru]^{a} & &
      \mathcal{B}(x,w);
    }
    \]
    that is, natural invertible \(2\)\nb-arrows,
    called \emph{associators},
    \[
    a(f_1,f_2,f_3)\colon(f_3\cdot f_2)\cdot
    f_1\underset{\simeq}{\Rightarrow} f_3\cdot(f_2\cdot f_1),
    \]
    where \(f_1\colon x\to y\),
    \(f_2\colon y\to z\) and \(f_3\colon z\to w\).
  \end{itemize}
  This data must make the following diagrams commute:
  \[
  \xymatrix{
    ((f_4\cdot f_3)\cdot f_2)\cdot f_1\ar@{=>}[r]^{} \ar@{=>}[d] &
    (f_4\cdot f_3)\cdot (f_2\cdot f_1)  \ar@{=>}[r]^{}&
    f_4\cdot (f_3\cdot(f_2\cdot f_1))\\
    (f_4\cdot (f_3\cdot f_2))\cdot f_1 \ar@{=>}[rr]^{}&&
    f_4\cdot ((f_3\cdot f_2)\cdot f_1) \ar@{=>}[u]^{},}
  \]
  \[
  \xymatrix{ (f_2\cdot1_y)\cdot f_1\ar@{=>}[rr]_{} \ar@{=>}[rrd]_{}& & f_2\cdot(1_y\cdot f_1)\ar@{=>}[d]^{} \\
    & &f_2\cdot f_1,}
  \]
  where \(f_1, f_2, f_3\),
  and~\(f_4\)
  are composable arrows, and the \(2\)\nb-arrows
  are associators and unitors and horizontal products of them with
  unit \(2\)\nb-arrows.
\end{defn}

We write ``\(\cdot\)''
or nothing for vertical products and ``\(\bullet\)''
for horizontal products.

\begin{example}
  \label{exa:Cat}
  Categories form a bicategory \(\mathbf{Cat}\)
  with functors as arrows and natural transformations as
  \(2\)\nb-arrows.
  Here the composition of morphisms is strictly associative and
  unital, that is, \(\mathbf{Cat}\) is even a \(2\)\nb-category.
\end{example}

\begin{example}
  \label{categoryexample}
  A category~\(\mathcal{C}\)
  may be regarded as a bicategory in which the categories
  \(\mathcal{C}(x,y)\) have only identity arrows.
\end{example}

\begin{example}
  The correspondence bicategory~\(\Corr\)
  is defined in~\cite{Buss-Meyer-Zhu:Higher_twisted} as the bicategory
  with \(\Cst\)\nb-algebras
  as objects, correspondences as arrows, and correspondence
  isomorphisms as \(2\)-arrows.
  The unit arrow~\(1_A\)
  on a \(\Cst\)\nb-algebra~\(A\)
  is~\(A\)
  viewed as a Hilbert \(A\)\nb-bimodule
  in the canonical way.  The \(A,B\)\nb-bimodule
  structure on~\(\Hilm[F]\)
  provides the unitors \(A\otimes_A\Hilm[F] \Rightarrow \Hilm[F]\)
  and \(\Hilm[F]\otimes_BB \Rightarrow \Hilm[F]\)
  for a correspondence \(\Hilm[F]\colon A\leadsto B\).
  The associators
  \((\Hilm\otimes_A \Hilm[F])\otimes_B \Hilm[G] \Rightarrow
  \Hilm\otimes_A (\Hilm[F]\otimes_B \Hilm[G])\)
  are the obvious isomorphisms.
\end{example}

\begin{defn}
  \label{homomorphismbicategory}
  Let \(\mathcal{B}, \mathcal{C}\)
  be bicategories.  A \emph{homomorphism}
  \(F\colon \mathcal{B}\to\mathcal{C}\) consists of
  \begin{itemize}
  \item a map \(F\colon \obj\mathcal{B}\to\obj\mathcal{C}\)
    between the object sets;

  \item functors
    \(F_{x,y}\colon \mathcal{B}(x,y)\to\mathcal{C}(F^0(x),F^0(y))\)
    for all \(x,y\in \obj\mathcal{B}\);

  \item natural transformations
    \[
    \xymatrix{
      \mathcal{B}(y,z)\times\mathcal{B}(x,y)\ar[rr]^{\circ} \ar[d]_{(F_{y,z},F_{x,y})}& &
      \mathcal{B}(x,z)\ar[d]^{F_{x,z}} \\
      \mathcal{C}(F(y),F(w))\times\mathcal{C}(F(x),F(y))
      \ar[rr]^{\circ}\ar@{=>}[rru]^{\varphi_{xyz}} & &
      \mathcal{C}(F(x),F(z))
    }
    \]
    for all triples \(x,y,z\)
    of objects of~\(\mathcal{B}\);
    explicitly, these are natural \(2\)\nb-arrows
    \(\varphi(f_1,f_2)\colon F_{y,z}(f_2)\cdot F_{x,y}(f_1)\Rightarrow
    F_{x,z}(f_2\cdot f_1)\);

  \item \(2\)\nb-arrows
    \(\varphi_x\colon 1_{F(x)}\Rightarrow F_{x,x}(1_x)\)
    for all objects~\(x\)
    of~\(\mathcal{B}\).
  \end{itemize}
  This data must make the following diagrams commute:
  \begin{equation}
    \begin{gathered}
      \label{associativityaxiom}
      \xymatrix{(F_{z,w}(f_3)\cdot F_{y,z}(f_2))\cdot F_{x,y}(f_1)
        \ar@{=>}[r]^{a'} \ar@{=>}[d]_{\varphi(f_2,f_3)\bullet1_{F_{x,y}(f_1) }}&
        F_{z,w}(f_3)\cdot (F_{y,z}(f_2)\cdot F_{x,y}(f_1))
        \ar@{=>}[d]^{1_{F_{z,w}(f_3) }\bullet\varphi(f_1,f_2)} \\
        F_{y,w}(f_3\cdot f_2)\cdot F_{x,y}(f_1) \ar@{=>}[d]_{\varphi(f_1,f_3\cdot f_2)}&
        F_{z,w}(f_3)\cdot F_{x,z}(f_2\cdot f_1) \ar@{=>}[d]_{\varphi(f_2\cdot f_1,f_3)}\\
        F_{x,w}((f_3\cdot f_2)\cdot f_1)\ar@{=>}[r]^{F_{x,w}(a)} &
        F_{x,w}(f_3\cdot (f_2\cdot f_1));
      }
    \end{gathered}
  \end{equation}
  \begin{equation}
    \begin{gathered}
      \label{rightunitaxiom}
      \xymatrix{ F_{x,y}(f_1)\cdot F_{x,x}(1_x)\ar@{=>}[rr]^{\varphi(1_x,f_1)} & &
        F_{x,y}(f_1\cdot 1_x)\ar@{=>}[d]^{F_{x,y}(r_{f_1})} \\
        F_{x,y}(f_1)\cdot 1_{F(x)}
        \ar@{=>}[rr]^{r'_{F_{x,y}(f_1)}}\ar@{=>}[u]^{1_{F_{x,y}(f_1)}\bullet\varphi_x} & &
        F_{x,y}(f_1);
      }
    \end{gathered}
  \end{equation}
  \begin{equation}
    \begin{gathered}
      \label{leftunitaxiom}
      \xymatrix{
        F_{y,y}(1_y)\cdot F_{x,y}(f_1) \ar@{=>}[rr]^{\varphi(f_1, 1_y)} & &
        F_{x,y}(1_y\cdot f_1)\ar@{=>}[d]^{F_{x,y}(l_{f_1})} \\
        1_{F(y)}\cdot F_{x,y}(f_1)
        \ar@{=>}[rr]^{l'_{F_{x,y}(f_1)}}\ar@{=>}[u]^{\varphi_y\bullet1_{F_{x,y}(f_1)}} & &
        F_{x,y}(f_1).
      }
    \end{gathered}
  \end{equation}
\end{defn}

\begin{example}
  A semigroup~\(P\)
  may be viewed a category with one object and~\(P\)
  as its set of arrows.  It may be viewed as a bicategory as well as
  in Example~\ref{categoryexample}.  A homomorphism from~\(P\)
  to \(\Corr\)
  is equivalent to an essential product system
  \((A, (\Hilm_p)_{p\in P^\op},\mu)\)
  over~\(P^\op\)
  as defined by Fowler~\cite{Fowler:Product_systems}.  The
  condition~\eqref{associativityaxiom} says that the multiplication
  maps
  \(\mu_{p,q}\colon
  \Hilm_p\otimes_A\Hilm_q\overset{\simeq}{\to}\Hilm_{q p}\)
  are associative.  The conditions \eqref{rightunitaxiom} and
  \eqref{leftunitaxiom} mean that
  \(\mu_{1,p}(a\otimes\xi)=\varphi_p(a)\xi\)
  and \(\mu_{p,1}(\xi\otimes a)=\xi a\)
  for \(a\in A\), \(\xi\in\Hilm_p\).
\end{example}

A morphism \(f\colon x\to y\)
in a bicategory~\(\mathcal{B}\)
induces functors
\[
f_*\colon\mathcal{B}(c,x)\to\mathcal{B}(c,y),\qquad
f^*\colon \mathcal{B}(y,c)\to\mathcal{B}(x,c)
\]
for \(c\in\obj \mathcal{B}\)
by composing arrows with~\(f\)
and composing \(2\)\nb-arrows
horizontally with~\(1_f\)
on one side (this is also called \emph{whiskering} with~\(f\)).

\begin{defn}
  \label{transfbicategories}
  Let \(F, G\colon\mathcal{B}\rightrightarrows\mathcal{C}\)
  be homomorphisms.  A \emph{transformation}
  \(\alpha\colon F\Rightarrow G\) consists of
  \begin{itemize}
  \item morphisms \(\alpha_x\colon F(x)\to G(x)\) for all
    \(x\in \obj\mathcal{B}\);

  \item natural
    transformations
    \[
    \xymatrix{
      \mathcal{B}(x,y)\ar[d]_{G_{x,y}} \ar[rr]^{F_{x,y}} & &
      \mathcal{C}(F(x),F(y))\ar[d]^{{\alpha_y}_*} \ar@{=>}[lld]_{\alpha_{x,y}}\\
      \mathcal{C}(G(x),G(y)) \ar[rr]^{{\alpha_x}^*} & &
      \mathcal{C}(F(x),G(y)),
    }
    \]
    that is, \(2\)\nb-arrows
    \(\alpha_{x,y}(f)\colon\alpha_y F_{x,y}(f)\Rightarrow
    G_{x,y}(f)\alpha_x\) for all \(x,y\in \obj\mathcal{B}\).
  \end{itemize}
  This data must make the following diagrams commute:
  \[
  \xymatrix{
    \alpha_z (F_{y,z}(g)F_{x,y}(f))  \ar@{<=>}[d]
    \ar@{=>}[r]^-{1\bullet\varphi_F(f,g)}&
    \alpha_z F_{x,z}(g f) \ar@{=>}[r]^{\alpha_{x,z}(gf)}&
    G_{x,z}(g f) \alpha_x\ar@{<=}[d]^{\varphi_G(f,g)\bullet1} \\
    (\alpha_z F_{y,z}(g))F_{x,y}(f) &&
    (G_{y,z}(g)G_{x,y}(f))\alpha_x \\
    (G_{y,z}(g)\alpha_y) F_{x,y}(f)
    \ar@{<=>}[r]^{}\ar@{<=}[u]^{\alpha_{y,z}(g)\bullet1} &
    G_{y,z}(g)(\alpha_y F_{x,y}(f))  \ar@{=>}[r]^{1\bullet\alpha_{x,y}(f)}&
    G_{y,z}(g)( G_{x,y}(f)\alpha_x) \ar@{<=>}[u]^{} .
  }
  \]
  \[
  \xymatrix{
    \alpha_x F_{x,x}(1_x)& &
    \alpha_x 1_{F(x)}  \ar@{=>}[ll]_{1_{\alpha_x}\bullet \varphi^F_x} \ar@{=>}[rr]^{r}&& \alpha_x\\
    G_{x,x}(1_x) \alpha_x
    \ar@{<=}[u]^{\alpha_{x,x}(1_x)} & & &&
    1_{G(x)} \alpha_x;  \ar@{=>}[llll]_{\varphi^G_x \bullet1_{\alpha_x}}
    \ar@{<=}[u]_{l^{-1}}
  }
  \]
\end{defn}

\begin{example}
  Let~\(G\)
  be a group.  A transformation between homomorphisms \(G\to\Corr\)
  consists of a correspondence \(\Hilm[F]\colon A\leadsto B\)
  and isomorphisms
  \(\alpha_s\colon\Hilm_s\otimes_A\Hilm[F]\simeq\Hilm[F]\otimes_B\Hilm[G]_s\)
  so that the following diagrams commute for all \(s,t\in G\):
  \[
  \xymatrix{
    (\Hilm_s\otimes_A\Hilm_t)\otimes_A\Hilm[F]  \ar@{<=>}[d]
    \ar@{=>}[rr]^{w^{1}_{s,t}\otimes1}& &
    \Hilm_{st}\otimes_A\Hilm[F]\ar@{=>}[rr]^{\alpha_{st}}&&
    \Hilm[F]\otimes_B\Hilm[G]_{st} \ar@{<=}[d]^{1\otimes w^2_{s,t}} \\
    \Hilm_s\otimes_A(\Hilm_t\otimes_A\Hilm[F]) && &&
    \Hilm[F]\otimes_B(\Hilm[G]_s\otimes_B\Hilm[G]_t)  \\
    \Hilm_s\otimes_A(\Hilm[F]\otimes_B\Hilm[G]_t)
    \ar@{<=>}[rr]^{}\ar@{<=}[u]^{1\otimes\alpha_t} & &
    (\Hilm_s\otimes_A\Hilm[F])\otimes_B\Hilm[G]_t  \ar@{=>}[rr]^{\alpha_s\otimes 1}&&
    (\Hilm[F]\otimes_B\Hilm[G]_s)\otimes_B\Hilm[G]_t \ar@{<=>}[u]^{} .
  }
  \]
  This is called a \emph{correspondence} of Fell bundles (see
  \cite{Buss-Meyer-Zhu:Higher_twisted}*{Proposition~3.23}).
\end{example}

\begin{defn}
  \label{def:modification}
  Let \(\alpha, \beta\colon F\Rightarrow G\)
  be transformations between homomorphisms.  A \emph{modification}
  \(\Delta\colon\alpha\Rrightarrow\beta\)
  is a family of \(2\)\nb-arrows
  \(\Delta_x\colon\alpha_x\Rightarrow\beta_x\)
  such that for every \(2\)\nb-arrow
  \(w\colon f_1\Rightarrow f_2\)
  for arrows \(f_1,f_2\colon x\to y\), the following diagram commutes:
  \[
  \xymatrix{
    \alpha_yF_{x,y}(f_1)\ar@{=>}[d]_{\alpha_{x,y}(f_1)}
    \ar@{=>}[rr]^{\Delta_y\bullet F_{x,y}(w)} & &
    \beta_yF_{x,y}(f_2)\ar@{=>}[d]^{\beta_{x,y}(f_2)}\\
    G_{x,y}(f_1)\alpha_x \ar@{=>}[rr]^{G_{x,y}(w)\bullet\Delta_x} & &
    G_{x,y}(f_2)\beta_x
  }
  \]
\end{defn}

\end{document}